\documentclass{amsart}

\usepackage{latexsym,amssymb,amsfonts, amsthm, amsmath}
\usepackage[english]{babel}
\usepackage{bm}
\usepackage{mathrsfs}

\usepackage{geometry}
\newtheorem{prop}{Proposition}[section]
\newtheorem{lem}[prop]{Lemma}
\newtheorem{cor}[prop]{Corollary}
\newtheorem{thm}[prop]{Theorem}
\newtheorem{conj}[prop]{Conjecture}
\theoremstyle{definition}
\newtheorem{rem}[prop]{Remark}
\newtheorem{defi}[prop]{Definition}
\newtheorem{ex}[prop]{Example}

\numberwithin{equation}{section}

\newcommand{\rev}{\mathop{\rm rev}\nolimits}

\def\Cay{\mathsf{Cay}}
\def\A{\mathcal{A}}
\def\B{\mathcal{B}}
\def\C{\mathcal{C}}

\def\Z{\mathbb{Z}}
\def\BHR{\mathop{\rm BHR}}
\def\eps{\varepsilon}
\def\equad{\quad \textrm{and} \quad}
\def\su{ \rightharpoonup}
\def\giu{\rightharpoonup}
\def\th{\vartheta}
\def\ol{\oplus}

\newcommand{\arq}[1]{\stackrel{#1}{\rightharpoondown}}
\newcommand{\art}[1]{\stackrel{#1}{\rightsquigarrow}}

\begin{document}

\title[New methods to attack the Buratti-Horak-Rosa conjecture]{New methods to attack \\the Buratti-Horak-Rosa conjecture}

\author{M.~A.~Ollis}
\address{Marlboro Institute for Liberal Arts and Interdisciplinary Studies,
Emerson College,
Boston, MA~02116, USA}
\email{matt\_ollis@emerson.edu}

\author{Anita Pasotti}
\address{DICATAM - Sez. Matematica, Universit\`a degli Studi di Brescia,
Via Branze 43,
I~25123 Brescia, Italy}
\email{anita.pasotti@unibs.it}

\author{Marco A. Pellegrini}
\address{Dipartimento di Matematica e Fisica, Universit\`a Cattolica del Sacro Cuore,
Via Musei 41,
I~25121 Brescia, Italy}
\email{marcoantonio.pellegrini@unicatt.it}

\author{John R. Schmitt}
\address{Mathematics Department, Middlebury College, Middlebury, VT~05753, USA}
\email{jschmitt@middlebury.edu}


\begin{abstract}
The conjecture, still widely open, posed by Marco Buratti, Peter Horak and Alex Rosa  states that a list $L$ of $v-1$ positive
integers not exceeding $\left\lfloor \frac{v}{2}\right\rfloor$ is the list of edge-lengths of a suitable
Hamiltonian path of the complete graph with vertex-set $\{0,1,\ldots,v-1\}$ if and only if,
for every divisor $d$ of $v$, the number of multiples of $d$ appearing in $L$
is at most $v-d$.
In this paper we present new methods that are based on linear realizations and can be applied to prove the validity of
this conjecture for a vast choice of lists. As example of their flexibility, we consider  lists whose underlying set
is one of the following: $\{x,y,x+y\}$, $\{1,2,3,4\}$, $\{1,2,4,\ldots,2x\}$, $\{1,2,4,\ldots,2x,2x+1\}$.
We also consider lists with many consecutive elements.
\end{abstract}

 \maketitle

\section{Introduction}

Throughout this paper $K_v$ denotes the complete graph whose vertex-set is $\{0,1,\dots,v-1\}$ for any positive integer $v$.
Following \cite{HR}, we define the \emph{length} $\ell(x,y)$ of an edge $[x,y]$ of $K_v$ as
$$\ell(x,y)=\min(|x-y|,v-|x-y|).$$
Given a subgraph $\Gamma$ of $K_v$, the list of edge-lengths of $\Gamma$ is the list
$\ell(\Gamma)$ of the lengths (taken with their respective multiplicities) of  all the edges
of $\Gamma$. For our convenience, if a list $L$ consists of
$a_1$ $1$s, $a_2$ $2$s, \ldots, $a_t$ $t$s,
we will write $L=\left\{1^{a_1},2^{a_2},\ldots,t^{a_t}\right\}$ and $|L|=\sum a_i$;
the set $\{i \mid a_i>0\}\subseteq L$  will be called the \emph{underlying set} of $L$.

The following conjecture was proposed in a private communication by Marco Buratti to Alex Rosa in 2007.
Buratti has never worked on it and he finally mentioned in 2013 (see \cite[p. 14]{BM}).

\begin{conj}[M. Buratti]
For any prime $p=2n+1$ and any list $L$ of $2n$ positive integers not exceeding $n$, there exists a
Hamiltonian path $H$ of $K_p$ with $\ell(H)=L$.
\end{conj}

The conjecture is almost trivially true in the case that $L$ has just one edge-length.
The case of exactly two distinct edge-lengths has been solved
independently in \cite{DJ, HR} and Mariusz Meszka checked by computer that the conjecture is true for all primes $p\leq 23$.
Francesco Monopoli in \cite{M} showed that the conjecture is true when all the elements of the list $L$ appear exactly
twice.

In \cite{HR} Peter Horak and Alex Rosa proposed a generalization of Buratti's conjecture, which has been
restated in an easier way in \cite{PP1} as follows.

\begin{conj}[P. Horak and A. Rosa]\label{HR iff B}
Let $L$ be a list of $v-1$ positive integers not exceeding $\left\lfloor\frac{v}{2}\right\rfloor$.
Then there exists a Hamiltonian path $H$ of $K_v$ such that $\ell(H)=L$ if, and only if, the following condition holds:
\begin{equation}\label{B}
\left. \begin{array}{c}
\textrm{for any divisor $d$ of $v$, the number of multiples of $d$} \\
\textrm{appearing in $L$ does not exceed $v-d$.}
\end{array}\right.
\end{equation}
\end{conj}

It is easy to see that in the statement of the conjecture of Horak and Rosa,
here denoted by $\BHR$, the actual conjecture is the sufficiency.
A list of $v-1$ positive integers not exceeding $\left\lfloor\frac{v}{2}\right\rfloor$ 
and satisfying the necessary condition \eqref{B} shall be said to be \emph{admissible}.
In \cite{HR} it was proved that the conjecture is true when the list $L$ has exactly two distinct elements.
Also, Mariusz Meszka checked by computer that the conjecture is true for all $v\leq 18$.
The case of exactly three distinct edge-lengths has been solved when these lengths are $1,2,3$ in \cite{CDF},
$1,2,5$ or $1,3,5$ or $2,3,5$ in \cite{PP1}, $1,2,4$ or $1,2,6$ or $1,2,8$ in \cite{PPnew}
and when  $L=\{1^a,2^b,t^c\}$ with $t$ even and $a+b\geq t-1$ in \cite{PPnew}.
The only case with four distinct edge-lengths for which the conjecture has been shown to be true is $1,2,3,5$, see \cite{PP1}.

Before giving the main results of this paper we would like to show some
connections between the $\BHR$ conjecture and other problems.
Mariusz Meszka proposed the following conjecture very similar, but easier, to Buratti's.

\begin{conj}[M. Meszka]
For any prime $p=2n+1$
and any list $L$ of $n$ positive integers not exceeding $n$, there exists a
near $1$-factor  $F$ of $K_p$ with $\ell(F)=L$.
\end{conj}

The same problem has been independently  proposed with a different terminology in \cite{B} and it has been completely solved by Emmanuel Preissmann and Maurice Mischler in \cite{PM}.
Also this conjecture has been generalized to the case in which the order of the complete graph is any odd integer, see
\cite{PPMeszka}.

\begin{conj}[A. Pasotti and M.A. Pellegrini]
Let $v=2n+1$ be an odd integer and $L$ be a list   of $n$ positive integers not exceeding $n$. Then there exists a
near $1$-factor  $F$ of $K_v$ with $\ell(F)=L$ if, and only if, the following condition holds:
\begin{equation*}
\left. \begin{array}{c}
\textrm{for any divisor $d$ of $v$, the number of multiples of $d$} \\
\textrm{appearing in $L$ does not exceed $\frac{v-d}{2}$.}
\end{array}\right.
\end{equation*}
\end{conj}

Some partial results have been obtained about this problem, see \cite{PPMeszka,R},
but the conjecture remains wide open. A similar conjecture for $v$ even has been
proposed in a private communication by Michal Adamaszek, see \cite{KP,KC}.

$\BHR$ conjecture is also related to many recent problems on partial sums, see \cite{AL, ADMS,CMPPSums, CP, HOS, O}.
Let $A$ be a finite list of elements of a group $(G,+)$.
Let $(a_1,a_2,\ldots,a_k)$ be an ordering of the elements in $A$ and define the \emph{partial sums}
$s_1,s_2,\ldots,s_k$ by the formula $s_j=\sum\limits_{i=1}^{j} a_i$ $(1\leq j \leq k)$.
It is not hard to see that, even if the
statement in terms of edge-lengths of a Hamiltonian path is more elegant, $\BHR$ conjecture can be formulated also in terms of partial sums of a given list
as follows:
\begin{quote}
Let $v$ be a positive integer and let $L$ be a list of $v-1$ nonzero elements of the cyclic group $(\Z_v,+)$. Then,
there exists a suitable sequence $(\eps_1,\ldots,\eps_{v-1})$, where each $\eps_i=\pm1$, and a suitable ordering
$(a_1,\ldots,a_{v-1})$ of $L$ such that the partial sums of the sequence $(\eps_1a_1,\ldots,\eps_{v-1}a_{v-1})$
are exactly the elements of $\Z_v\setminus\{0\}$ if and only if for any divisor $d$ of $v$ the number of multiples of 
$[d]_v \in \Z_v$ appearing in $L$
does not exceed $v-d$.
\end{quote}
As shown in \cite{PP1}, $\BHR$ conjecture is also related to cyclic decompositions, in fact it can be reformulated as follows.
We refer to \cite{BP} for the concepts necessary to understand the relationship.

\begin{conj}
A Cayley multigraph $\Cay[\Z_v:\Lambda]$ admits a cyclic decomposition into Hamiltonian paths if and only if
$\Lambda=L\cup -L$ with $L$ satisfying condition \eqref{B}.
\end{conj}

We can therefore conclude that $\BHR$ conjecture fits into a wide range of (open) problems and we believe that it is
worthy of being studied  for this reason.

The main aim of this paper is to provide new techniques that can be applied to a vast range of lists.
In fact, in Section \ref{sec:realization} we recall the concepts of cyclic and linear realization and then
we show how linear realizations of two lists $L_1,L_2$ can be manipulated in order to get a linear realization
of $L_1\cup L_2$.
We also introduce two particular classes of linear realizations that allow to build inductive processes.
To give an idea of the flexibility of these new constructions, in Sections \ref{sec:xodd}, \ref{sec:other} and \ref{new}
we obtain the following results. By  $\BHR(L)$ we mean  Conjecture \ref{HR iff B} for a given list $L$.

\begin{thm}\label{main}
Let $x\geq 3$ be an integer and let $L=\{1^a,x^b,(x+1)^c\}$ be an admissible list.
Then $\BHR(L)$ holds in each of the following cases:
\begin{itemize}
\item[(1)] $x$ is odd and $a\geq \min\{3x-3, b+2x-3 \}$;
\item[(2)] $x$ is odd, $a\geq 2x-2$ and $c\geq \frac{4}{3} b$;
\item[(3)] $x$ is even and $a\geq \min\{3x-1, c+2x-1\}$;
\item[(4)] $x$ is even, $a\geq 2x-1$ and $b\geq c$.
\end{itemize}
\end{thm}

\begin{thm}\label{pari}
Let $x\geq 2$ and  $c\geq 1$. Let
$$L=\left\{1^{a}, 2^{b_2}, 4^{b_4}, 6^{b_6},\ldots, \ell^{b_\ell},  x^c \right\}$$
be an admissible list, where $\ell=2\left\lfloor\frac{x-1}{2}\right\rfloor$.
Then $\BHR(L)$ holds in each of the following cases:
\begin{itemize}
 \item[(1)] $x$ is even and $a\geq x-1$;
 \item[(2)] $x$ is odd and $a\geq 3x-4$.
\end{itemize}
\end{thm}

\begin{thm}\label{th1234}
Let $L=\{1^a, 2^b, 3^c, 4^d\}$ be an admissible list, where $c,d\geq 1$.
Then $\BHR(L)$ holds for all $a\geq 3$ and all $b\geq 0$. Also,  $\BHR(L)$ holds when $a=2$ and $b\geq 1$.
\end{thm}

Finally, in the last section we investigate lists whose underlying set consists of many consecutive elements.

\begin{thm}\label{66}
Let $L = \{ 1^{a_1}, \ldots, m^{a_m} \}$ be an admissible list.
Then $\BHR(L)$ holds whenever
$a_1 \geq a_2 \geq \cdots \geq a_m > 0$.
\end{thm}

\section{Some new methods to  work with linear realizations}\label{sec:realization}

In this section we define cyclic and linear realizations of a list $L$ and
show how they are useful to prove $\BHR(L)$.
A \emph{cyclic realization} of a list $L$ with $v-1$ elements each
from the set
$\{1,\ldots,\lfloor\frac{v}{2}\rfloor\}$ is a Hamiltonian path $[x_0,x_1,\ldots,x_{v-1}]$
of $K_{v}$ such that the list of edge-lengths
$\{\ell(x_i,x_{i+1})\mid i=0,\ldots,v-2\}$ equals $L$. So it is clear that
$\BHR(L)$ can be so reformulated:
 every such a list $L$ has a cyclic realization
if and only if condition \eqref{B} is satisfied.
For example, the path $[0,5,11,3,9,1,8,2,10,4,12,6,7]$ is a cyclic realization of $\{1,5^5,6^6\}$.

Now, let $L$ be a list with $v-1$ positive integers not exceeding
$v-1$.
A \emph{linear realization} of  $L$, denoted by $rL$, is a Hamiltonian path
$[x_0,x_1,\ldots,x_{v-1}]$ of $K_v$
such that $L=\{|x_i-x_{i+1}|\ |\ i=0,\ldots,v-2\}$.
By \emph{standard} linear realization, we mean a linear realization starting with $0$.
For instance, one can easily check that the path $[0,4,1,5,2,3,7,11,8,12,9,6,10,13,17,14,18,15,$ $16,19]$ is a standard
linear realization
of $\{1^2,3^9,4^8\}$.

\begin{rem}\label{cyclin}
Every linear realization of a list $L$ can be viewed as a
cyclic realization of a suitable list $L'$ but not necessarily of the same list. For example, the path
$[0,1,2,7,3,4,5,6]$
 is a linear realization of $L=\{1^5,4,5\}$ and a cyclic realization of $L'=\{1^5,3,4\}$.
Anyway, if all the elements in the list are less than or equal to $\left\lfloor\frac{|L|+1}{2}\right\rfloor$, then every
linear realization of $L$ is
also a cyclic realization of the same list $L$ (see \cite[Section 3]{HR}).
\end{rem}

\subsection{Concatenating two standard linear realizations}

If $v = 2m+1$, let $L^* = \{ 1^2, 2^2, \ldots,$ $m^2 \}$; if $v=2m$, let $L^* = \{ 1^2, 2^2, \ldots, (m-1)^2, m^1 \}$.
Realizations of $L^*$ are also known as \emph{terraces} for $\Z_v$ and have been well-studied; see, for example,~\cite{OllisSurvey}.
The oldest, and most well-known, terrace is the \emph{Walecki Construction} (see~\cite{Alspach08} for more of its history and uses):
$$ [ 0, v-1, 1, v-2,2, \ldots, \lfloor v/2 \rfloor ].$$

Additionally, this is a standard linear realization of $\{ 1, \ldots, v-1 \}$.  Linear realizations of $\{ 1, \ldots, v-1 \}$ are  known as
\emph{graceful permutations} or \emph{graceful labelings of paths}; see~\cite{GallianSurvey} for more on the theory of graceful
labelings.

Now starting from a linear realization $\bm{g}$ we construct three kinds of linear realizations related to $\bm{g}$
as follows.
Let $\bm{g} = [g_1, g_2, \ldots, g_v]$ be a linear realization of a list $L$.  The
\emph{reverse}, $$\rev(\bm{g}) = [g_v, g_{v-1}, \ldots, g_1]$$ is also a linear realization of $L$,
as is its {\em complement}, $$\bm{\bar{g}} = [v-1-g_1,  v-1-g_2, \ldots, v -1 - g_v].$$
The \emph{translation} of  $\bm{g}$  by~$a$ is
$\bm{g} + a = [g_1+ a,  \ldots, g_v+ a ]$, which has the same absolute differences as~$\bm{g}$.

Theorem \ref{th:2lr} generalizes a method from the constructions of Horak and Rosa \cite{HR} that we shall use more widely. 

\begin{thm}\label{th:2lr}
Let~$L_1$ and~$L_2$ be lists.  If each of~$L_1$ and $L_2$ has a standard linear realization, then $L = L_1 \cup L_2$ has a linear realization.
\end{thm}

\begin{proof}
Let $\bm{g} = [g_1, \ldots, g_m]$ and $\bm{h} = [h_1, \ldots, h_n]$ be linear realizations of~$L_1$ and~$L_2$
respectively with $g_1 = 0 = h_1$.  Consider the sequence $\bm{g}\ol \bm{h}$ obtained by concatenating
$\rev(\bm{ \bar{g}} )$ and $[h_2, \ldots, h_n] + (m-1)$, which has length~$m+n-1$.

The first $m-1$ absolute differences give the elements of~$L_1$ and, as the $m$th element of the sequence is $m-1 = h_1 + m-1$, the remaining $n -1$ absolute differences give the elements of~$L_2$.  Hence the sequence is a linear realization of~$L$.
\end{proof}

\begin{ex}
Let $\bm{g}=[0,5,3,6,4,1,2]$ and $\bm{h}=[0,3,7,1,5,2,6,4]$ be standard linear realizations of
$L_1=\{1,2^2,3^2,5\}$ and $L_2=\{2,3^2,4^3,6\}$ respectively.
Then $\bm{ \bar{g}}=[6,1,3,0,2,5,4]$ and
$$\bm{g}\ol\bm{h}=[4,5,2,0,3,1,6, 9,13,7,11,8,12,10  ]$$
is a linear realization of $L_1\cup L_2=\{ 1,2^3,3^4,4^3,5,6\}$.
\end{ex}

Note that, in general, the realization constructed in the proof of~Theorem~\ref{th:2lr} is not standard and so cannot be used inductively.
On the other hand some kinds of linear
realizations can be used in an inductive construction, as we are going to show.
Following the terminology  introduced in \cite{CDF}, we will say that a standard linear realization of a list $L$ is \emph{perfect},
and we denote it by $RL$,
if the terminal vertex of the path is labeled by the largest element.
For instance, $[0,1,2,\ldots,a]$ is a perfect linear realization of $\{1^a\}$.
Given a perfect realization
$RL_1=[0,x_1,\ldots,x_{s-1},s]$
and a standard linear realization
$rL_2=[0,y_1,\ldots,y_t]$, not necessarily perfect, we may
form the standard linear realization of
$L_1\cup L_2$ denoted by $RL_1+rL_2$ so defined
$$RL_1+rL_2=[0,x_1,\ldots,x_{s-1},s,y_1+s,\ldots,y_t+s].$$
Note that $RL_1+RL_2$ is perfect.
It is important to underline that the previous construction, in general, does not work if
we consider \emph{cyclic} realizations.

Horak and Rosa~\cite[Theorem 3.1]{HR} prove that $\BHR(L)$ holds for
$L = \{ 1^{a_1}, 2^{a_2}, \ldots, m^{a_m} \}$ where $a_i \leq 2$ for $i >1$  when~$v$ is an odd prime.
This was accomplished by showing that it has a linear realization that is also the appropriate cyclic realization.
We generalize the method of proof used in~\cite{HR}, which already works for all odd~$v$, to give Corollary~\ref{cor:hr_big1},
which includes an analogous result for even~$v$.

\begin{cor}\label{cor:hr_big1}
Let $M^*$ be a sublist of $L^*$ with $|M^*| = a$.  Then $\BHR(L)$ holds for $L = (L^* \setminus M^*) \cup \{ 1^a \}$.
\end{cor}

\begin{proof}
Let $x,y$ be the two largest elements of $L$, possibly $x=y$. Write $L = L_1 \cup L_2 \cup L_3$ where:
\begin{itemize}
 \item $L_1 = \{ 1^{a_1}, 2^{a_2}, \ldots, x^{a_x} \}$, with $a_1 = x - \sum\limits_{i=2}^x a_i$ and $a_i \leq 1$ for $i>1$,
 \item $L_2 = \{ 1^{b_1}, 2^{b_2}, \ldots, y^{b_y} \}$, with $b_1 = y - \sum\limits_{i=2}^y b_i$ and $b_i \leq 1$ for $i>1$, and 
 \item $L_3 = \{ 1^{a - a_1 - b_1} \}$.
\end{itemize}

Horak and Rosa~\cite[Lemma~3.13]{HR} show that there are two standard linear realizations $rL_1$ and $rL_2$.
Now, taking a perfect linear realization $RL_3$, we obtain the linear realization $rL_1 \ol ( RL_3+ rL_2)$ of $L$.
\end{proof}

\subsection{Inserting elements to linear realizations of special type}

We finally introduce some particular kinds of linear realizations which turn out to be very useful in the
following sections.

\begin{defi}\label{def:AxBx}
Let $x$ be a positive integer. We say that a linear realization $rL$ of a list $L$ is:
\begin{itemize}
\item of type $\A_x$ if the two vertices $|L|-x$ and $|L|-x+1$ are adjacent in $rL$;
\item of type $\B_x$ if the two vertices $|L|-x$ and $|L|$  are adjacent in $rL$.
\end{itemize}
\end{defi}

Clearly, $\A_1=\B_1$.
The importance of these  types of linear realizations is explained in the following.
We define a function $\eta_x$ acting on linear realizations of type $\A_x$ as follows:
the path $\eta_x(rL)$ is obtained from $rL$ by inserting $|L|+1$ between the adjacent vertices $|L|-x$ and
$|L|-x+1$. It is easy to see that $\eta_x(rL)$ is a linear realization of $(L\setminus \{1\})\cup \{x,x+1 \}$.
Analogously, we define a function $\mu_x$ acting on linear realizations of type $\B_x$ in the following way:
the path $\mu_x(rL)$ is obtained from $rL$ by inserting $|L|+1$ between the adjacent vertices $|L|-x$ and
$|L|$. Note that  $\mu_x(rL)$ is a linear realization of $(L\setminus\{x\})\cup \{1, x+1\}$.
Furthermore, if we apply the function $\eta_x$  or $\mu_x$ to a standard linear realization, we obtain a linear realization
which is still standard.

\begin{ex}
Let $L=\{1^5, 7^3,8\}$ and take $rL=[0,7,8, 1,9, 2,3,4,5,6]$.
Note that $rL$ is of type $\A_7$, since the vertices $|L|-7=2$ and $|L|-7+1=3$ are adjacent, and it is also of type $\B_7$, since the vertices $|L|-7=2$ and $|L|=9$ are adjacent.
Hence, starting from $rL$ and applying the functions $\eta_7$ and $\mu_7$, we obtain, for instance:
$$\begin{array}{rcccl}
r\{1^4,7^4,8^2\} & =& \eta_7(rL) & = &[0,7,8, 1,9, 2,\bm{10},3,4,5,6],\\
r\{1^3,7^5,8^3\} & =& \eta_7(\eta_7(rL))& = &[0,7,8, 1,9, 2,10,3,\bm{11},4,5,6],\\
r\{1^6,7^2, 8^2 \} & = & \mu_7(rL) & = &[0,7,8, 1,9, \bm{10}, 2,3,4,5,6],\\
r\{1^5,7^3, 8^3 \} & =& \eta_7(\mu_7(rL)) & =&[0,7,8, 1,9, 10, 2,3,\bm{11},4,5,6],\\
r\{1^6,7^2, 8^4 \} & =& \mu_7(\eta_7(\mu_7(rL))) & = &[0,7,8, 1,9, 10, 2,3,11,\bm{12},4,5,6],\\
r\{1^5,7^3, 8^5 \} & =& \eta_7(\mu_7(\eta_7(\mu_7(rL)))) & = &[0,7,8, 1,9, 10, 2,3,11,12,4,5,\bm{13},6].\\
\end{array}$$
\end{ex}

\begin{defi}\label{def:Cx}
Let $x$ be a positive even integer. We say that a linear realization $rL$ of a list $L$ is
of type $\C_x$  if the vertices $|L|-(2i+1)$ and  $|L|-2i$ are adjacent for all
$i=0,\ldots,\frac{x-2}{2}$.
\end{defi}

It is also immediate that, if $rL$ is of type $\C_x$, then it is of type $\C_{y}$ for all even integers $2\leq y\leq x$.

\begin{prop}\label{prop:sigmax}
If $rL$ is a linear realization of a list $L$ of type $\C_x$, for some even $x$, then
\begin{itemize}
\item[(1)] $\th_x(rL)=\mu_{x-1} \circ \eta_{x-1}(rL)$ is a linear realization of $L\cup \{x^2\}$ of type $\C_x$;
\item[(2)] $\sigma_x(rL)=\mu_x \circ \eta_{x-1}(rL)$ is a linear realization of $L \cup\{x-1,x+1\}$ of type $\C_x$.
\end{itemize}
\end{prop}
\begin{proof}
(1) Let $rL$ be a linear realization of a given list $L$ of type $\C_x$.
Then, in particular, $|L|-(x-1)$ and $|L|-(x-2)$ are adjacent in $rL$ and hence it is also a linear realization of type $\A_{x-1}$
so we can apply $\eta_{x-1}$. In this way we get a linear realization of $(L\setminus\{1\})\cup\{x-1,x\}$
that is of type $\B_{x-1}$ since $|L|+1$ and $|L|+1-(x-1)$ are adjacent.
So we can apply $\mu_{x-1}$
obtaining a linear realization of $L\cup\{x^2\}$ which is still of type $\mathcal{C}_x$.

(2) As before, if $rL$ is a linear realization of a given list $L$ of type $\C_x$,
 we can apply $\eta_{x-1}$ and, in this way, we get a linear realization of $(L\setminus\{1\})\cup\{x-1,x\}$
that is of type $\B_x$ since $|L|+1$ and $|L|+1-x$ are adjacent. So we can apply $\mu_x$
obtaining a linear realization of $L\cup\{x-1,x+1\}$. Note that since $x$ is even, we are adding to $L$ two odd consecutive positive integers.
\end{proof}

\begin{ex}
We construct a  linear realization of  $\{1^3,2,3,4^5,5^2, 6^3,7\}$.
Start with the list $L=\{1^3,2,4,6\}$ and its linear realization
$rL=[0,6,5,1,2,4,3]$  of type $\C_6$. Then
$$\begin{array}{rlcl}
L_1=\{1^3,2,4,6^3\}: & rL_1=\vartheta_6(rL) & = & [0,6,5,1,\bm{7},\bm{8}, 2,4,3], \\
L_2=\{1^3,2,4,5,6^3,7\}: & rL_2=\sigma_6 (rL_1) & = &  [0,6, 5,1,7, 8, 2,4,\bm{9},\bm{10},3],\\
L_3=\{1^3,2,4^5,5, 6^3,7\}: & rL_3=\vartheta_4^2(rL_2) & = &   [0,6, 5,1,7,\bm{11},\bm{12}, 8, 2,4, 9,\bm{13},\bm{14}, 10 ,3],\\
\end{array}$$
and then
$$r\{1^3,2,3,4^5,5^2, 6^3,7\}
=\sigma_4(L_3) =[0,6, 5,1,7,11,\bm{16},\bm{15},12, 8, 2,4, 9, 13,14, 10 ,3] .$$
\end{ex}

The constructions we described in this section are very flexible, as we are going to show in the remaining parts of this paper.
In particular, they can be applied for lists whose underlying set is one of the following:
$\{x,y,x+y\}$,  $\{1,2,4,6,\ldots,2x\}$, $\{1,2,4,6,\ldots,2x, 2x+1\}$, $\{1,2,3,\ldots,x\}$.

\section{On the lists whose underlying set is $\{1,x,x+1\}$}\label{sec:xodd}

The results of this section are an example of application of the linear realizations and the functions
$\eta_x$ and  $\mu_x$ introduced in Section \ref{sec:realization}.
In particular, we start considering $\BHR$ conjecture for lists $L=\{1^a,x^b,(x+1)^c\}$, where $x\geq 3$.
Note that if $x=1$, the underlying set of $L$ has size $2$ and hence the validity of $\BHR(L)$ has  already been proved in \cite{HR}.
Also, the case $x=2$  has already been solved in \cite{CDF}.

We will make use of the following perfect linear realizations.

\begin{lem}\label{perfette}
There exists a perfect linear realization of the following lists:
\begin{itemize}
\item[(1)] $\{1^{x-1},x^{qx}\}$ for any odd integer $x$ and any integer $q\geq 0$;
\item[(2)] $\{1^{x-1},x^2,(x+1)^{x-1}\}$ for any even integer $x$;
\item[(3)] $\{x^x, (x+1)^{x+1}\}$ for any integer $x$.
\end{itemize}
\end{lem}

\begin{proof}
(1) Fix $x$ and $q\geq 0$: for any integer $i\geq 0$ we denote by $i\su qx+i$ the sequence
$i, x+i, 2x+i,\ldots, qx+i$ and by $qx+i \giu i$ the sequence $qx+i, (q-1)x+i,\ldots, x+i, i$.
Then
$$\begin{array}{rcl}
R\{1^{x-1},x^{qx}\} & =& [0\su qx,qx+1 \giu 1, 2 \su qx+2, qx+3 \giu 3,\ldots, x-3 \su qx+x-3, \\
&&  qx+x-2\giu x-2, x-1\su qx+x-1].
    \end{array}$$
(2) and (3) It suffices to take
$R\{1^{x-1},x^2,(x+1)^{x-1}\}=[0,x+1,x+2, 1,2, x+3,x+4, 3,4,\ldots, x-5,x-4, 2x-3,2x-2,
x-3,x-2, 2x-1,x-1,x,2x ]$ for
$x$ even, and
$R\{x^x, (x+1)^{x+1}\}=[0,x+1,1, x+2,2, x+3,3, \ldots, 2x-1, x-1, 2x, x, 2x+1 ]$.
\end{proof}

\subsection{Case $x$ odd}

In this subsection we investigate $\BHR$ conjecture for lists $\{1^a,x^b,(x+1)^c\}$, where $x\geq 3$ is an odd integer.
We also set $\eta=\eta_x$ and  $\mu=\mu_x$.

\begin{lem}\label{odd0}
Let $x\geq 3$ be an odd integer. Then, for all $c\geq 0$, there exists:
\begin{itemize}
\item[(1)] a standard linear realization of type $\A_x$ of the list $\{1^x, (x+1)^c\}$;
\item[(2)] a standard linear realization of type $\B_x$ of the list $\{1^{x-1},x, (x+1)^c\}$.
\end{itemize}
\end{lem}
\begin{proof}
Start by considering the following linear realization of $L=\{1^x\}$:
$R L=[0,1,2,\ldots,x]$.
Then, $RL $ is of type $\A_x$, in fact the vertices $|L|-x=0$ and $|L|-x+1=1$ are adjacent, and so we can apply the
function $\eta$:
$$\eta(RL)=[0, \bm{x+1}, 1, 2, 3,\ldots,x],$$
which is a linear realization  of $L'=\{1^{x-1},x, x+1\}$ of type $\B_x$, since the vertices $|L'|-x=1$ and
$|L'|=x+1$ are adjacent.
We now apply the function $\mu$:
$$(\mu\circ\eta)(RL )=[0, x+1, \bm{x+2},1, 2, 3,\ldots,x],$$
obtaining a linear realization  of $L''=\{1^x, (x+1)^2\}$ of type $\A_x$, since $|L''|-x=2$ and $|L''|-x+1=3$ are
adjacent vertices.
Applying alternatively the functions $\eta$ and $\mu$ we produce the following linear realizations of type $\A_x$ and $\B_x$, respectively:
$$r\{1^x, (x+1)^{2\bar c} \}=(\mu \circ \eta)^{\bar c} (RL)\equad r\{1^{x-1}, x, (x+1)^{2\bar
c+1}\}=(\eta\circ(\mu\circ\eta)^{\bar c} )(RL )$$
for all $\bar c\geq 0$.

Consider now the following linear realization of $L=\{1^{x-1},x\}$:
$r L=[0,x,x-1,\ldots,2,1]$.
This is a linear realization of type $\B_x$, as the vertices $|L|-x=0$ and $|L|=x$ are adjacent. So, we can apply the
function $\mu$:
$$\mu(r L)=[0, \bm{x+1}, x, x-1,\ldots,3,2,1],$$
which is a linear realization  of $L'=\{1^{x},x+1 \}$ of type $\A_x$, since the vertices $|L'|-x=1$ and
$|L'|-x+1=2$ are adjacent. We now apply the function $\eta$:
$$(\eta\circ\mu)(r L )=[0, x+1,x,x-1,\ldots,5,4,3,2,\bm{x+2},1],$$
obtaining a linear realization of $L''=\{1^{x-1},x, (x+1)^2\}$ of type $\B_x$ , since $|L''|-x=2$ and $|L''|=x+2$ are
adjacent vertices.
Applying alternatively the functions $\mu$ and $\eta$ we get the following linear realizations of type
$\A_x$ and $\B_x$, respectively:
$$r\{1^x, (x+1)^{2\bar c+1} \}=(\mu\circ(\eta\circ\mu)^{\bar c}) (rL)\equad r\{1^{x-1}, x, (x+1)^{2\bar
c}\}=(\eta\circ\mu)^{\bar c} (rL )$$
for all $\bar c\geq 0$.
\end{proof}

\begin{ex}
We show how it is possible to obtain a  linear realization of type $\A_3$ of $\{1^3,4^5\}$ and
a linear realization of type $\B_3$ of $\{1^2,3,4^6\}$ following the proof of Lemma \ref{odd0}.
Start with the list $L=\{1^2,3\}$ and its linear realization $rL=[0,3,2,1]$ of type $\B_3$. Then
$$\begin{array}{rcccl}
r\{1^3,4\} & = & \mu_3(rL) & =& [0,\bm{4},3,2,1], \\
r\{1^2,3,4^2\} & = & (\eta_3\circ\mu_3)(rL) & =& [0,4,3,2,\bm{5},1], \\
r\{1^3,4^3\} & = & (\mu_3\circ\eta_3\circ\mu_3)(rL) & =& [0,4,3,2,\bm{6}, 5,1], \\
r\{1^2,3,4^4\} & = & (\eta_3\circ\mu_3)^2(rL) & =& [0,4,\bm{7},3,2,6,5,1], \\
r\{1^3,4^5\} & = & (\mu_3\circ(\eta_3\circ\mu_3)^2)(rL) & =& [0, 4,\bm{8},7,3,2,6,5,1], \\
r\{1^2,3,4^6\} & = & (\eta_3\circ\mu_3)^3(rL) & =& [0,4,8,7,3,2,6,\bm{9},5,1].
\end{array}$$
\end{ex}

\begin{lem}\label{odd1}
Let $x\geq 3$ be an odd integer and let $s$ be an even integer such that $2\leq  s\leq x-1$. Then there exists:
\begin{itemize}
\item[(1)] a standard linear realization of type $\A_x$ of the list $\{1^{x-1}, x^s, (x+1)^c\}$ for all even $c\geq 0$;
\item[(2)] a standard linear realization of type $\B_x$ of the list $\{1^{x-2}, x^{s+1}, (x+1)^c\}$ for all odd $c\geq 1$.
\end{itemize}
\end{lem}

\begin{proof}
For $i\geq 1$, let $U_i$ be the sequence $x+2i-2, x+2i-1, 2i-1, 2i$ and  consider the following linear
realization of $L=\{1^{x-1}, x^s\}$, where $2\leq s\leq x-1$ is even:
$$rL=[0,\; U_1,U_2,\ldots,U_{\frac{s}{2}-1}, U_{\frac{s}{2}},\; s+1, s+2, s+3,\ldots,x-1].$$
This realization is  of type $\A_x$, since the vertices $|L|-x = s-1$ and $|L|-x + 1=s$ are adjacent (both in $U_\frac{s}{2}$).
Hence, we can apply the function $\eta$:
$$ \eta(rL ) = [0,\; U_1,\ldots,U_{\frac{s}{2}-1},\; x+s-2, x+s-1, s-1, \bm{x+s}, s,s+1,s+2,\ldots,x-1].$$
Note that by the very particular structure of $L$ it follows that $\eta(rL)$ is a linear realization
of $L'=\{1^{x-2}, x^{s+1}, x+1\}$ of type $\B_x$, in fact the vertices $|L'|-x=s$ and
$|L'|=x+s$ are adjacent.
Hence, we can apply the function $\mu$, obtaining:
$$\begin{array}{rcl}
(\mu\circ\eta)(rL ) &= & [0,\; U_1,\ldots,U_{\frac{s}{2}-1},\; x+s-2, x+s-1, s-1, x+s, \bm{x+s+1},s,\\
&&s+1,s+2,\ldots,x-1].
\end{array}$$
This is a linear realization  of $L''=\{1^{x-1}, x^{s}, (x+1)^2\}$ of type $\A_x$, since
the vertices $|L''|-x=s+1$ and $|L''|-x+1=s+2$ are adjacent.
Applying alternatively the functions $\eta$ and $\mu$ we obtain the following linear realizations of type $\A_x$ and $\B_x$, respectively:
$$r\{1^{x-1}, x^s, (x+1)^{2\bar c} \}=(\mu \circ \eta)^{\bar c} (rL)\equad r\{1^{x-2}, x^{s+1}, (x+1)^{2\bar
c+1}\}=(\eta\circ(\mu\circ\eta)^{\bar c} )(rL )$$
for all $\bar c\geq 0$ and all even $2\leq s\leq x-1$.
\end{proof}

\begin{ex}
We construct a linear realization of  type $\A_5$ of $\{1^4,5^4,6^6\}$ and
a linear realization of type $\B_5$ of $\{1^3,5^5,6^7\}$ following the proof of Lemma \ref{odd1}, Cases (1) and (2), respectively.
Start with the list $L=\{1^4,5^4\}$ and its linear realization $rL=[0,5,6,1,2,7,8,3,4]$ of type $\A_5$. Then
$$\begin{array}{rcccl}
r\{1^3,5^5,6\} & = & \eta_5(rL) & =& [0,5,6,1,2,7,8,3,\bm{9},4], \\
r\{1^4,5^4,6^2\} & = & (\mu_5\circ\eta_5)(rL) & =& [0,5,6,1,2,7,8,3,9,\bm{10},4], \\
r\{1^3,5^5,6^3\} & = & (\eta_5\circ\mu_5\circ\eta_5)(rL) & =& [0,5,\bm{11}, 6,1,2,7,8,3,9,10,4], \\
r\{1^4,5^4,6^4\} & = & (\mu_5\circ\eta_5)^2(rL) & =& [0,5,11,\bm{12},6,1,2,7,8,3,9,10,4], \\
r\{1^3,5^5,6^5\} & = & (\eta_5\circ(\mu_5\circ\eta_5)^2)(rL) & =& [0, 5,11,12,6, 1,2, 7,\bm{13},8,3,9,10,4], \\
r\{1^4,5^4,6^6\} & = & (\mu_5\circ\eta_5)^3(rL) & =& [0, 5,11,12,6, 1,2, 7,13,\bm{14},8,3,9,10,4], \\
 r\{1^3,5^5,6^7\} & = & (\eta_5\circ(\mu_5\circ\eta_5)^3)(rL) & =& [0, 5,11,12,6, 1,2, 7,13,14,8,3,9,\bm{15},10,4].
\end{array}$$
\end{ex}

\begin{lem}\label{odd2}
Let $x\geq 3$ be an odd integer. Then there exists:
\begin{itemize}
\item[(1)] a standard linear realization of type $\B_x$ of the list $\{1^{2x-s}, x^s, (x+1)^c\}$ for all odd $c\geq 1$
and all even $2\leq  s\leq x-1$;
\item[(2)] a standard linear realization of type $\A_x$ of the list $\{1^{2x+1-s}, x^{s-1}, (x+1)^c\}$ for all even $c\geq
2$  and all even $2\leq  s\leq x-1$;
\item[(3)] a standard linear realization of $\{1^{x-1 }, x^s\}$ for all odd $1\leq  s\leq x$.
\end{itemize}
\end{lem}

\begin{proof}
For $i\geq 1$, let $V_i$ be the sequence $2x+2-2i, 2x+1-2i, x+1-2i, x-2i$ and  consider the following linear
realization of $L=\{1^{2x-s}, x^s, x+1\}$, where $2\leq s\leq x-1$ is even:
$$\begin{array}{rcl}
rL & =& [0,x,x+1, 2x+1,\;  V_1,V_2,\ldots,V_{\frac{s}{2}-1},\;2x-s+2,2x-s+1,2x-s,\ldots, x+2,\\
&& 1,2,3,\ldots,x+1-s].
  \end{array}$$
This realization is of type $\B_x$, since the vertices $|L|-x = x+1$ and $|L|=2x+1$ are adjacent. Hence,
we  can apply the function $\mu$:
$$\begin{array}{rcl}
\mu(rL) & =& [0,x,x+1,\bm{2x+2}, 2x+1,\;  V_1,V_2,\ldots,V_{\frac{s}{2}-1},\;2x-s+2,2x-s+1,\\
&& 2x-s,\ldots, x+2, 1,2,3,\ldots,x+1-s].
  \end{array}$$
Note that by the very particular structure of $L$ it follows that $\mu(rL)$ is a linear realization
 of $L'=\{1^{2x+1-s}, x^{s-1}, (x+1)^2\}$ of type $\A_x$, in fact the vertices $|L'|-x=x+2$ and
 $|L'|-x+1=x+3$ are adjacent.
Hence, we can apply the function $\eta$, obtaining:
 $$\begin{array}{rcl}
 (\eta\circ\mu)(rL ) &= &  [0,x,x+1,2x+2, 2x+1,\;  V_1,V_2,\ldots,V_{\frac{s}{2}-1},\;2x-s+2,2x-s+1,\\
&& 2x-s,\ldots, x+3, \bm{2x+3}, x+2, 1,2,3,\ldots,x+1-s].
 \end{array}$$
This is a linear realization  of $L''=\{1^{2x-s}, x^{s}, (x+1)^3\}$ of type $\B_x$, since
the vertices $|L''|-x=x+3$ and $|L''|=2x+3$ are adjacent.
Applying alternatively the functions $\mu$ and $\eta$ we obtain the following linear realizations of type
$\B_x$ and $\A_x$, respectively:
$$r\{1^{2x-s}, x^s, (x+1)^{2\bar c+1} \}=(\eta \circ \mu)^{\bar c} (rL)\equad r\{1^{2x+1-s}, x^{s-1}, (x+1)^{2\bar
c+2}\}=(\mu\circ(\eta\circ\mu)^{\bar c} )(rL )$$
for all $\bar c\geq 0$ and all even $2\leq s\leq x-1$. This proves (1) and (2).

Finally, to prove (3),  define for all $i\geq 1$ the sequence
$W_i=2i,x+2i,x+2i-1, 2i-1$.
Then, for all odd $1\leq  s\leq x$, we obtain
$$r\{1^{x-1},x^s\}=[0,x,x-1,\ldots,s+1,s,W_{\frac{s-1}{2}},\ldots,W_2,W_1].$$
\end{proof}

\begin{ex}
Following the proof of Lemma \ref{odd2}, Cases (1) and (2), it is possible to obtain a
linear realization  of $\{1^{8},7^6,8^7\}$ of  type $\B_7$ and
a linear realization of $\{1^{9},7^5,8^8\}$ of type $\A_7$, respectively.
Start with the list $L=\{1^8,7^6,8\}$ and its linear realization
$rL=[0,7,8,15,14,13,6,5,12,11,4,3,10,9,1,2]$  of type $\B_7$. Then

\begin{footnotesize}
$$\begin{array}{rcccl}
r\{1^9,7^5,8^2\} & = & \mu_7(rL) & =& [0,7,8,\bm{16},15,14,13,6,5,12,11,4,3,10,9,1,2], \\
r\{1^8,7^6,8^3\} & = & (\eta_7\circ\mu_7)(rL) & =& [0,7,8,16,15,14,13,6,5,12,11,4,3,10,\bm{17},9,1,2], \\
r\{1^9,7^5,8^4\} & = & (\mu_7\circ\eta_7\circ\mu_7)(rL) & =& [0,7,8,16,15,14,13,6,5,12,11,4,3,10,\bm{18},17,9,1,2], \\
r\{1^8,7^6,8^5\} & = & (\eta_7\circ\mu_7)^2(rL) & =& [0,7,8,16,15,14,13,6,5,12,\bm{19},11,4,3,10,18,17,9,1,2], \\
r\{1^9,7^5,8^6\} & = & (\mu_7\circ(\eta_7\circ\mu_7)^2)(rL) & =&
[0,7,8,16,15,14,13,6,5,12,\bm{20},19,11,4,3,10,18,17,9,1,2], \\
r\{1^8,7^6,8^7\} & = & (\eta_7\circ\mu_7)^3(rL) & =&  [0,7,8,16,15,14,\bm{21},13,6,5,12,20,19,11,4,3,10,18,17,9,\\
&&&&1,2], \\
r\{1^9,7^5,8^8\} & = & (\mu_7\circ(\eta_7\circ\mu_7)^3)(rL) & =&
[0,7,8,16,15,14,\bm{22},21,13,6,5,12,20,19,11,4,3,10,18,17,\\
&&&& 9,1,2] .
\end{array}$$
\end{footnotesize}
\end{ex}

\begin{prop}\label{cor:odd}
Let $x\geq 3$ be an odd integer. Then there exists a standard linear realization of $\{1^{2x-2}, x^s, (x+1)^c\}$
for all $c\geq 0$ and all $0\leq s\leq x-1$.
\end{prop}

\begin{proof}
Using the perfect realization $R\{1^{\bar a} \}=[0,1,\ldots,\bar a]$, for a suitable $\bar a\geq 0$, we have a linear realization of
\begin{itemize}
\item[(1)] $\{1^{2x-2}, (x+1)^c \}$ for all $c\geq 0$ by Lemma \ref{odd0}(1);
\item[(2)] $\{1^{2x-2}, x^s, (x+1)^c \}$ for all $c\geq 0$  and all odd $1\leq s\leq x-2$ by Lemma \ref{odd0}(2), Lemma
\ref{odd1}(2)
and Lemma \ref{odd2}(2),(3);
\item[(3)] $\{1^{2x-2}, x^s, (x+1)^c \}$ for all $c\geq 0$  and all even  $2\leq s\leq x-1$ by Lemma
\ref{odd1}(1) and Lemma \ref{odd2}(1).
\end{itemize}
\end{proof}

\begin{thm}\label{prop:odd}
Let $x\geq 3$ be an odd integer. Then there exists a standard linear realization of $\{1^a,x^b,$ $(x+1)^c\}$ in each of the
following cases:
\begin{itemize}
\item[(1)] $b,c\geq 0$ and  $a\geq 3x-3$;
\item[(2)] $b,c\geq 0$ and  $a\geq b+2x-3$;
\item[(3)] $b\geq 0$, $c\geq \frac{4b}{3}$ and $a\geq 2x-2$.
\end{itemize}
\end{thm}

\begin{proof}
If $b=0$ the thesis follows from Lemma \ref{odd0}(1). So fix $b>0$ and $c\geq 0$,  and write $b=qx+s$ with $0\leq s\leq x-1$.
Let $\bm{g}$ be the linear realization of $\{1^{2x-2}, x^s, (x+1)^c \}$ whose existence is proved in Proposition
\ref{cor:odd}. By Lemma \ref{perfette}(1) we obtain
$$r\{1^{3x-3},  x^{b}, (x+1)^c \}=R\{1^{x-1}, x^{qx} \} + r\{1^{2x-2}, x^s, (x+1)^c \}.$$
Hence, there exists a linear realization
$$r\{1^{a}, x^{b},(x+1)^c \} =R\{1^{a+3-3x}\}+ r\{1^{3x-3}, x^{b}, (x+1)^c \}$$
for all $a\geq 3x-3$.
Furthermore, we also get
$$r\{1^{q(x-1)+ 2x-2 },  x^{b },(x+1)^c \}=\underbrace{R\{1^{x-1},x^x\} +\ldots+R\{1^{x-1},x^x\}}_{q \textrm{
times}}+\bm{g}.$$
This proves the existence of a linear realization
$$r\{1^{a}, x^{b},(x+1)^c \} =R\{1^{a+2-q(x-1)-2x}\}+ r\{1^{q(x-1)+2x-2}, x^{b},(x+1)^c \}$$
for all $a\geq b+2x-3$, since in this case $a \geq q(x-1)+2x-2$.

Now, for every $\bar c\geq 0$, let $\bm{h_{\bar c}}$ be the linear realization of $\{1^{2x-2}, x^{s}, (x+1)^{\bar c} \}$,
whose
existence is given by Proposition \ref{cor:odd}.
Then we obtain
$$r\{1^{2x-2},x^{b},  (x+1)^{q(x+1)+\bar c} \}=\underbrace{R\{x^x,(x+1)^{x+1}\} +\ldots+R\{x^x,(x+1)^{x+1} \}}_{q
\textrm{
times}}+ \bm{h_{\bar c}}.$$
If $a\geq 2x-2$ and  $c\geq \frac{4b}{3}$, then we can write $c=q(x+1)+\bar c$, for a suitable $\bar c$, and  take
$$r\{1^{a}, x^{b}, (x+1)^c \} =R\{1^{a+2-2x}\}+ r\{1^{2x-2},  x^{b}, (x+1)^c \}.$$
\end{proof}

\subsection{Case $x$ even}\label{sec:xeven}

In this subsection we investigate $\BHR$ conjecture for lists $\{1^a,x^b,(x+1)^c\}$, where $x\geq 4$ is an even integer.
We also set $\eta=\eta_{x-1}$ and  $\mu=\mu_{x-1}$.

\begin{lem}\label{even1}
Let $x\geq 4$ be an even integer. Then there exists a standard linear realization of type $\A_{x-1}$
of the list $\{1^{x+1}, x^b, (x+1)^s\}$ for all even $b\geq 0$  and all $0\leq  s\leq x$.
\end{lem}

\begin{proof}
For $i\geq 1$, let $U_i$ be the sequence $x+2i, x+1+2i, 2i, 2i+1$ and  consider the following linear
realization of $L=\{1^{x+1}, (x+1)^s\}$, where $0\leq s\leq x$ is even:
$$rL=[0,1,\; U_1, U_2,\ldots,U_{\frac{s}{2}},\; s+2, s+3, s+4,\ldots,x,x+1].$$
This realization is  of type $\A_{x-1}$, since the vertices $|L|-(x-1) = s+2$ and $|L|-(x-1) + 1=s+3$ are adjacent. Hence, we
can apply the function $\eta$:
$$\eta(rL) = [ 0,1,\; U_1, U_2,\ldots, U_{\frac{s}{2}},\; s+2,\bm{x+s+2}, s+3, s+4,\ldots,x,x+1  ].$$
Note that by the very particular structure of $L$ it follows that $\eta(rL)$ is a linear realization
 of $L'=\{1^{x}, x-1, x, (x+1)^s\}$ of type $\B_{x-1}$, as the vertices $|L'|-(x-1)=s+3$ and
$|L'|=x+s+2$ are adjacent.
Hence, we can apply the function $\mu$, obtaining:
$$\begin{array}{rcl}
(\mu\circ\eta)(rL) &= & [ 0,1,\; U_1, U_2,\ldots, U_{\frac{s}{2}},\; s+2,x+s+2, \bm{x+s+3}, s+3,\\
&& s+4,s+5,\ldots,x,x+1  ].
\end{array}$$
This is a linear realization of $L''=\{1^{x+1}, x^{2}, (x+1)^s\}$ of type $\A_{x-1}$, since
the vertices $|L''|-(x-1)=s+4$ and $|L''|-(x-1)+1=s+5$ are adjacent.
Applying alternatively the functions $\eta$ and $\mu$ we obtain the following linear realization of type $\A_{x-1}$:
$$r\{1^{x+1}, x^{2\bar b}, (x+1)^{s} \}=(\mu \circ \eta)^{\bar b} (rL)$$
for all $\bar b\geq 0$ and all even $0\leq s\leq x$.

Now, for $i\geq 1$, let $V_i$ be the sequence $2i+2, 2i+1, x+2+2i, x+1+2i$ and  consider the following linear
realization of $L=\{1^{x+1}, (x+1)^s\}$, where $1\leq s\leq x-1$ is odd:
$$rL=[0,1,\; x+2,x+1,\ldots,s+4,s+3,s+2,\;  V_{\frac{s-1}{2}}, \ldots,V_2,V_1,\; 2].$$
This realization is  of type $\A_{x-1}$, since the vertices $|L|-(x-1) = s+2$ and $|L|-(x-1) + 1=s+3$ are adjacent. Hence, we
can apply the function $\eta$:
$$\eta(rL) = [0,1,\; x+2,x+1,\ldots,s+4,s+3,\bm{x+s+2}, s+2,\;  V_{\frac{s-1}{2}}, \ldots,V_2,V_1,\; 2  ].$$
Note that by the very particular structure of $L$ it follows that $\eta(rL)$ is a linear realization
 of $L'=\{1^{x}, x-1, x, (x+1)^s\}$ of type $\B_{x-1}$, as the vertices $|L'|-(x-1)=s+3$ and
$|L'|=x+s+2$ are adjacent.
Hence, we can apply the function $\mu$, obtaining:
$$\begin{array}{rcl}
(\mu\circ\eta)(rL) &= & [0,1,\; x+2,x+1,\ldots,s+5,s+4,s+3,\bm{x+s+3},x+s+2, s+2,\\
&& V_{\frac{s-1}{2}}, \ldots,V_2,V_1,\; 2 ].
\end{array}$$
This is a linear realization  of $L''=\{1^{x+1}, x^{2}, (x+1)^s\}$ of type $\A_{x-1}$, since
the vertices $|L''|-(x-1)=s+4$ and $|L''|-(x-1)+1=s+5$ are adjacent.
Applying alternatively the functions $\eta$ and $\mu$ we obtain the following linear realization of type $\A_{x-1}$:
$$r\{1^{x+1}, x^{2\bar b}, (x+1)^{s} \}=(\mu \circ \eta)^{\bar b} (rL)$$
for all $\bar b\geq 0$ and all odd $1\leq s\leq x-1$.
\end{proof}

\begin{ex}
We construct a  linear realization of  type $\A_3$ of $\{1^{5},4^6,5^3\}$
following the proof of Lemma \ref{even1}.
Start with the list $L=\{1^5,5^3\}$ and its linear realization
$rL=[0,1,6,5,4,3,8,7,2]$  of type $\A_3$. Then
$$\begin{array}{rcccl}
r\{1^4,3,4,5^3\} & = & \eta_3(rL) & =& [0,1,6,\bm{9},5,4,3,8,7,2], \\
r\{1^5,4^2,5^3\} & = & (\mu_3\circ\eta_3)(rL) & =& [0,1,6,\bm{10},9,5,4,3,8,7,2], \\
r\{1^4,3,4^3,5^3\} & = & (\eta_3\circ \mu_3\circ \eta_3)(rL) & =& [0,1,6,10,9,5,4,3,8,\bm{11},7,2], \\
r\{1^5,4^4,5^3\} & = & (\mu_3\circ\eta_3)^2(rL) & =& [0,1,6,10,9,5,4,3,8,\bm{12},11,7,2], \\
r\{1^4,3,4^5,5^3\} & = & (\eta_3\circ(\mu_3\circ\eta_3)^2)(rL) & =& [0,1,6,10,\bm{13},9,5,4,3,8,12,11,7,2], \\
r\{1^5,4^6,5^3\} & = & (\mu_3\circ\eta_3)^3(rL) & =& [0,1,6,10,\bm{14},13,9,5,4,3,8,12,11,7,2].
\end{array}$$
\end{ex}

\begin{lem}\label{even2}
Let $x\geq 4$ be an even integer. Then there exists:
\begin{itemize}
\item[(1)] a standard linear realization of type $\A_{x-1}$ of the list $\{1^{x-1}, x^b \}$ for all $b\geq 0$;
\item[(2)] a standard linear realization of type $\A_{x-1}$ of the list $\{1^{2x-1-s}, x^b, (x+1)^s\}$ for all odd $b\geq 1$
and all even $2\leq s \leq x-2$;
\item[(3)] a standard linear realization of type $\A_{x-1}$ of the list $\{1^{x}, x^b, (x+1)^s\}$ for all odd $b\geq 1$
and all odd $1\leq  s\leq x-1$;
\item[(4)] a standard linear realization of the list $\{1^{x}, x, (x+1)^x \}$ and a linear realization of the list
$\{1^{2x-1}, x^b,(x+1)^x \}$ for all odd $b\geq 3$.
\end{itemize}
\end{lem}

\begin{proof}
(1) follows from Lemma \ref{odd0}(1).\\
(2) For $i\geq 1$, let $U_i$ be the sequence $x+2i-1,  x+2i,  2i-1, 2i$ and  consider the following linear
realization of $L=\{1^{2x-1-s}, x, (x+1)^s\}$, where $2\leq s\leq x-2$ is even:
$$rL=[0,\; U_1, U_2,\ldots,U_{\frac{s}{2}},\;s+1,s+2,\ldots, x-1,x,\; 2x,2x-1,\ldots, x+s+1].$$
This realization is of type $\A_{x-1}$, since the vertices $|L|-(x-1) = x+1$ and $|L|-(x-1) + 1=x+2$ are adjacent
(both in $U_1$).  Hence, we  can apply the function $\eta$:
$$\begin{array}{rcl}
\eta(rL) & =&  [0,\; x+1, \bm{2x+1}, x+2, 1,2, U_2,\ldots,U_{\frac{s}{2}},\;s+1,s+2,\ldots, x-1,x, \\
&&2x,2x-1,\ldots, x+s+1    ].
\end{array}$$
Note that by the very particular structure of $L$ it follows that $\eta(rL)$ is a linear realization
 of $L'=\{1^{2x-2-s},x-1, x^2, (x+1)^s\}$ of type $\B_{x-1}$, as the vertices $|L'|-(x-1)=x+2$ and $|L'|=2x+1$ are
adjacent. Hence, we can apply the function $\mu$, obtaining:
 $$\begin{array}{rcl}
 (\mu\circ\eta)(rL) &= & [ 0,\; x+1, 2x+1,\bm{2x+2}, x+2, 1,2, U_2,\ldots,U_{\frac{s}{2}},\;s+1,s+2,\ldots, x-1,\\
 && x, 2x,2x-1,\ldots, x+s+1  ].
\end{array}$$
This is a linear realization of $L''=\{1^{2x-1-s}, x^3, (x+1)^s\}$  of type $\A_{x-1}$, since
 the vertices $|L''|-(x-1)=x+3$ and $|L''|-(x-1)+1=x+4$ are adjacent (both in $U_2$ or both at the end of
$(\mu\circ\eta)(rL)$ if $s=2$).
 Applying alternatively the functions $\eta$ and $\mu$ we obtain the following linear realization of type $\A_{x-1}$:
$$r\{ 1^{2x-1-s}, x^{2\bar b+1}, (x+1)^s  \}=(\mu \circ \eta)^{\bar b} (rL)$$
 for all $\bar b\geq 0$ and all even $2\leq s\leq x-2$. \\
(3) Now, for $i\geq 1$, let $V_i$ be the sequence $2i,2i-1, x+2i,x+2i-1$ and  consider the following linear
realization of $L=\{1^{x},x, (x+1)^s\}$, where $1\leq s\leq x-1$ is odd:
$$rL=[0,\; x, x-1,x-2,\ldots,s+3, s+2, \;  V_{\frac{s+1}{2}}, \ldots,V_2,V_1].$$
This realization is  of type $\A_{x-1}$, since the vertices $|L|-(x-1) = s+2$ and $|L|-(x-1) + 1=s+3$ are adjacent.
Hence, we can apply the function $\eta$:
$$\eta(rL) = [0,\; x, x-1,x-2,\ldots,s+3, \bm{x+s+2}, s+2, \;  V_{\frac{s+1}{2}}, \ldots,V_2,V_1   ].$$
Note that by the very particular structure of $L$ it follows that $\eta(rL)$ is a linear realization
 of $L'=\{1^{x-1}, x-1, x^2, (x+1)^s\}$ of type $\B_{x-1}$, as the vertices $|L'|-(x-1)=s+3$ and
$|L'|=x+s+2$ are adjacent.
Hence, we can apply the function $\mu$, obtaining:
$$\begin{array}{rcl}
(\mu\circ\eta)(rL) &= & [0,\; x, x-1,x-2,\ldots,s+5, s+4, s+3, \bm{x+s+3}, x+s+2, s+2,\\
&&  V_{\frac{s+1}{2}}, \ldots,V_2,V_1  ].
\end{array}$$
This is a linear realization  of $L''=\{1^{x}, x^{3}, (x+1)^s\}$ of type $\A_{x-1}$, since
the vertices $|L''|-(x-1)=s+4$ and $|L''|-(x-1)+1=s+5$ are adjacent.
Applying alternatively the functions $\eta$ and $\mu$ we obtain the following linear realization of type $\A_{x-1}$:
$$r\{1^{x}, x^{2\bar b+1}, (x+1)^{s} \}=(\mu \circ \eta)^{\bar b} (rL)$$
for all $\bar b\geq 0$ and all odd $1\leq s\leq x-1$. \\
(4) Take the sequence $U_i$ previously defined. Then
$$[0,U_1,U_2,\ldots,U_{\frac{x}{2}-1},2x-1, x-1, x, 2x+1, 2x]$$
is a linear realization of $\{1^{x}, x, (x+1)^x \}$.
Hence, for all odd $b\geq 3$ we have
$$r\{1^{2x-1} , x^b, (x+1)^x\}= R\{1^{x-1}, x^2, (x+1)^{x-1} \} + r\{1^{x}, x^{b-2}, x+1 \},$$
where the existence of $R\{1^{x-1}, x^2, (x+1)^{x-1}\}$ is proved in Lemma \ref{perfette}(2) and
the existence of $r\{1^{x}, x^{b-2}, x+1 \}$ is showed in the previous item.
\end{proof}

\begin{ex}
We construct a  linear realization of  type $\A_3$ of $\{1^{5},4^7,5^2\}$
following the proof of Lemma \ref{even2}, Case (2).
Consider the list $L=\{1^5,4,5^2\}$ and its linear realization
$rL=[0,5,6,1,2,3,4,8,7]$  of type $\A_3$. Then
$$\begin{array}{rcccl}
r\{1^4,3,4^2,5^2\} & = & \eta_3(rL) & =& [0,5,\bm{9},6,1,2,3,4,8,7], \\
r\{1^5,4^3,5^2\} & = & (\mu_3\circ\eta_3)(rL) & =& [0,5,9,\bm{10},6,1,2,3,4,8,7], \\
r\{1^4,3,4^4,5^2\} & = & (\eta_3\circ \mu_3\circ \eta_3)(rL) & =& [0,5,9,10,6,1,2,3,4,8,\bm{11},7], \\
r\{1^5,4^5,5^2\} & = & (\mu_3\circ\eta_3)^2(rL) & =& [0,5,9,10,6,1,2,3,4,8,\bm{12},11,7], \\
r\{1^4,3,4^6,5^3\} & = & (\eta_3\circ(\mu_3\circ\eta_3)^2)(rL) & =& [0,5,9,\bm{13},10,6,1,2,3,4,8,12,11,7], \\
r\{1^5,4^7,5^2\} & = & (\mu_3\circ\eta_3)^3(rL) & =& [0,1,6,10,13,\bm{14},9,5,4,3,8,12,11,7,2].
\end{array}$$
\end{ex}

\begin{ex}
We construct a  linear realization of  type $\A_5$ of $\{1^{6},6^9,7^3\}$
following the proof of Lemma \ref{even2}, Case (3).
Consider the list $L=\{1^6,6,7^3\}$ and its linear realization
$rL=[0,6,5,4,3,10,9,2,1,8,7]$  of type $\A_5$. Then

\begin{footnotesize}
$$\begin{array}{rcccl}
r\{1^5,5,6^2,7^3\} & = & \eta_5(rL) & =& [0,6,\bm{11},5,4,3,10,9,2,1,8,7], \\
r\{1^6,6^3,7^3\} & = & (\mu_5\circ\eta_5)(rL) & =& [0,6,\bm{12},11,5,4,3,10,9,2,1,8,7], \\
r\{1^5,5,6^4,7^3\} & = & (\eta_5\circ \mu_5\circ \eta_5)(rL) & =&  [0,6,12,11,5,4,3,10,9,2,1,8,\bm{13},7], \\
r\{1^6,6^5,7^3\} & = & (\mu_5\circ\eta_5)^2(rL) & =& [0,6,12,11,5,4,3,10,9,2,1,8,\bm{14},13,7], \\
r\{1^5,5,6^6,7^3\} & = & (\eta_5\circ(\mu_5\circ\eta_5)^2)(rL) & =& [0,6,12,11,5,4,3,10,\bm{15},9,2,1,8,14,13,7], \\
r\{1^6,6^7,7^3\} & = & (\mu_5\circ\eta_5)^3(rL) & =& [0,6,12,11,5,4,3,10,\bm{16},15,9,2,1,8,14,13,7],\\
r\{1^5,5,6^8,7^3\} & = & (\eta_5\circ(\mu_5\circ\eta_5)^3)(rL) & =& [0,6,12,\bm{17},11,5,4,3,10,16,15,9,2,1,8,14,13,7], \\
r\{1^6,6^9,7^3\} & = & (\mu_5\circ\eta_5)^4(rL) & =& [0,6,12,\bm{18},17,11,5,4,3,10,16,15,9,2,1,8,14,13,7].
\end{array}$$
\end{footnotesize}
\end{ex}

From Lemmas \ref{even1} (for $b$ even) and \ref{even2} (for $b$ odd), and arguing as in the proof of
Proposition \ref{cor:odd},  we obtain the following.

\begin{prop}\label{cor:even}
Let $x\geq 4$ be an even integer. Then there exists a standard linear realization of $\{1^{2x-1}, x^b, (x+1)^s\}$
for all $b\geq 0$ and all $0\leq s\leq x$.
\end{prop}

\begin{thm}\label{prop:even}
Let $x\geq 4$ be an even integer. Then there exists a standard linear realization of $\{1^a,x^b,(x+1)^c\}$
in each of the following cases:
\begin{itemize}
\item[(1)] $b,c\geq 0$ and  $a\geq 3x-1$;
\item[(2)] $b,c\geq 0$ and  $a\geq c+2x-1$;
\item[(3)] $b\geq c\geq 0$ and $a\geq 2x-1$.
\end{itemize}
\end{thm}

\begin{proof}
Fix $b,c\geq 0$ and write $c=q(x+1)+s$ with $0\leq s\leq x$.
Let $\bm{g}$ be the linear realization  of $\{1^{2x-1}, x^b, (x+1)^s \}$ whose existence is proved in Proposition
\ref{cor:even}.
By Lemma \ref{perfette}(1), we
obtain
$$r\{1^{3x-1},  x^{b}, (x+1)^c \}=R\{1^{x}, (x+1)^{q(x+1)} \} + \bm{g}.$$
Hence, there exists a linear realization
$$r\{1^{a}, x^{b},(x+1)^c \} =R\{1^{a-3x+1}\}+ r\{1^{3x-1}, x^{b}, (x+1)^c \}$$
for all $a\geq 3x-1$. Furthermore,
we have
$$r\{1^{(q+2)x-1},  x^{b},(x+1)^{c} \}=\underbrace{R\{1^{x},(x+1)^{x+1}\}
+\ldots+R\{1^{x},(x+1)^{x+1}\}}_{q \textrm{ times}}+\bm{g}.$$
So,  we obtain a linear realization
$$r\{1^{a}, x^{b},(x+1)^c \} =R\{1^{a+1-(q+2)x}\}+ r\{1^{(q+2)x-1}, x^{b},(x+1)^c \}$$
for all $a\geq c+2x-1$, since in this case $a\geq (q+2)x-1$.

Now, for every $\bar b\geq 0$, let $\bm{h_{\bar b}}$ be the linear realization of $\{1^{2x-1}, x^{\bar b}, (x+1)^s\}$, whose
existence is given by Proposition \ref{cor:even}.
Then we obtain
$$r\{1^{2x-1 },x^{qx+\bar b},  (x+1)^{c} \}=\underbrace{R\{x^x,(x+1)^{x+1}\} +\ldots+R\{x^x,(x+1)^{x+1} \}}_{q \textrm{
times}}+\bm{h_{\bar b}},$$
where the existence of $R\{x^x,(x+1)^{x+1}\}$ is proved in Lemma \ref{perfette}(3).
If $a\geq 2x-1$ and  $b\geq c$, then we can write $b=qx+\bar b$, for a suitable $\bar b$, and  take
$$r\{1^{a}, x^{b},(x+1)^c \} =R\{1^{a-2x+1}\}+ r\{1^{2x-1},  x^{b}, (x+1)^c \}.$$
\end{proof}

\subsection{Main result}\label{Concl}

We are now ready to prove our main result for this kind of list.

\begin{proof}[Proof of Theorem \rm{\ref{main}}]
Let $x\geq 3$, $L=\{1^a,x^b,(x+1)^c\}$ with $a,b,c\geq 0$ and set $v=a+b+c+1$.
Since $x+1\leq  \left \lfloor \frac{v}{2}\right\rfloor$, every linear realization of $L$ is also a cyclic realization of the same list,
see Remark \ref{cyclin}.
Hence, the validity of $\BHR$ follows from Theorems \ref{prop:odd} and \ref{prop:even}.
\end{proof}

To conclude this section, we explain how starting from this result it is possible to obtain
lists of the form $\{y^a,z^b,(y+z)^c\}$ for which $\BHR$ conjecture holds.
Let $a,b,c\geq 0$, $v=a+b+c+1$ and $1\leq y,z,t\leq \left \lfloor \frac{v}{2}\right\rfloor$
such that $t\equiv y+z\pmod v$. If $\gcd(y,v)=1$ we can multiply by $y^{-1}$ all the elements
of the list $L'=\{y^a,z^b,t^c\}$, where $y^{-1}$ denotes the unique integer $w$ such that $1\leq w \leq v-1$
and $yw\equiv 1\pmod v$. We then obtain the list $\bar{L}=\{1^a,x^b,(x+1)^c\}$, where $x\equiv y^{-1}z\pmod v$ and $1\leq x \leq v-1$.
Let $L$ be the list obtained from $\bar{L}$ by replacing each integer $i> \left \lfloor \frac{v}{2}\right\rfloor$ in
$\bar{L}$,
with $v-i$.
The validity of $\BHR(L)$ for the cases described in Theorem \ref{main} implies the validity of $\BHR(L')$
for the same choice of $a,b,c$ since, multiplying the vertices of a Hamiltonian path of $K_v$ by an element coprime with
$v$, we obtain again a Hamiltonian path of $K_v$, for details see \cite[Lemma 2]{DJ}.
Clearly, the same argument holds if $\gcd(z,v)=1$.
In particular, if $v$ is a prime, namely if we consider the original Buratti's conjecture,
both conditions $\gcd(y,v)=1$ and $\gcd(z,v)=1$ are satisfied.

For instance, consider the list $L'=\{13^{20}, 28^{3}, 41^{65} \}$.
Then $v=89$ and multiplying by $48$ we get the list
$L=\{1^{20}, 9^{3}, 10^{65}\}$.
By Theorem \ref{main}(2), $\BHR(L)$ holds and so  $\BHR(L')$ also holds.
Now,  consider the list $L'=\{13^{20}, 19^{80}, 32^{10} \}$.
Then $v=111$ and multiplying by $94$ we get the list
$L=\{1^{20}, 10^{80}, 11^{10}\}$.
Applying Theorem \ref{main}(4), $\BHR(L)$ holds and then also $\BHR(L')$ holds.

\begin{cor}
Let $a,b,c\geq 0$, $v=a+b+c+1$ and $1\leq y,z,t\leq \left \lfloor \frac{v}{2}\right\rfloor$
such that $t\equiv y+z\pmod v$. Suppose that $\gcd(y,v)=1$ or $\gcd(z,v)=1$.
If $a,b,c$ satisfy one of the conditions given in Theorem \rm{\ref{main}}, then $\BHR(L)$ holds for the
list $L=\{y^a,z^b,t^c\}$.
\end{cor}

\section{On the lists with almost all even elements}\label{sec:other}

In this section we consider lists whose elements are all even integers plus the element $1$, and at most another odd integer.
Our constructions are based on linear realizations of type $\C_x$.
Here, given two nonnegative integers $x,y$, we denote
by $x\su y$ the sequence $x,x+1,x+2,\ldots,y-1, y$ if $x\leq y$, and the sequence
$x,x-1,x-2,\ldots,y+1,y$, otherwise.

\begin{prop}\label{all_even}
Let $t\geq 2$ be an even integer and suppose that $c\geq 1$.
For every choice of $b_2,b_4,\ldots,b_{t-2}\geq 0$
there exists a standard linear realization of
$$L=\left\{1^{t-1}, 2^{b_2}, 4^{b_4}, 6^{b_6},\ldots, (t-2)^{b_{t-2}},  t^c \right\}.$$
\end{prop}

\begin{proof}
Suppose that the exponents $b_{x_1},b_{x_2},\ldots,b_{x_k}$ are all odd integers, while the exponents $b_{y_1},
b_{y_2},\ldots, b_{y_h}$ are all even.
Write $L= L_1 \cup L_2\cup\{1^{k}\}$, where
$$L_1 =\left\{1^{t-k-1}, x_1,x_2, \ldots, x_k,t^c\right\}\equad L_2=\{2^{2\bar b_2}, 4^{2\bar
b_4},\ldots,(t-2)^{2\bar b_{t-2}} \}.$$
It will be convenient (and not restrictive) to assume $x_1>x_2>\ldots>x_k$. Define $S_0=0$ and $S_j=\sum\limits_{i=1}^j (-1)^{i+1} x_i$
for $j=1,\ldots,k$. We split the proof into two cases according to the parity of $c$.

Suppose that $c$ is odd. Write $c=2\bar c+1$ and define $L_1'=\left\{1^{t-k-1}, x_1,x_2, \ldots, x_k,t\right\}$.
We construct the following linear realizations of $L_1'$:
if $k=0$, take $r L_1' =[0,t\giu 1]$;
if $k>0$ is even, take
$$\begin{array}{rcl}
rL_1'& =&[0, t\giu 1+S_1,  1+S_0\su S_2, S_1\giu 1+S_3,\ldots,1+S_{k-4}\su S_{k-2},\\
&&S_{k-3}\giu 1+S_{k-1},  1+S_{k-2}\su S_k, S_{k-1}\giu 1+S_k];
\end{array}$$
if $k$ is odd, take
$$\begin{array}{rcl}
rL_1'& =&[0, t\giu 1+S_1,  1+S_0\su S_2, S_1\giu 1+S_3,\ldots,1+S_{k-3}\su S_{k-1}, S_{k-2}\giu 1+S_{k}, \\
&& 1+S_{k-1}\su S_k].
\end{array}$$
Note that, in all cases, the vertices $2i+1$ and $2i+2$ are adjacent for all $i=0,\ldots,\frac{t-2}{2}$.
This implies that $rL_1'$ is of type $\C_t$. So, by Proposition \ref{prop:sigmax}(1), we can apply the function $\th_t^{\bar c}$, obtaining
$rL_1=(\th_t)^{\bar c}(rL_1')$.

Next, suppose that $c$ is even. We construct the following linear realizations of $L_1'=\left\{1^{t-k-1}, x_1,x_2,
\ldots, x_k\right\}$:
if $k=0$, take $r L_1' =[0\su t-1]$;
if $k>0$ is even, take
$$\begin{array}{rcl}
rL_1'& = & [0\su t-1-S_1,  t-1-S_0\giu t-S_2, t-S_1\su t-1-S_3,\ldots, \\
&& t-1-S_{k-4} \giu t- S_{k-2}, t-S_{k-3}\su t-1-S_{k-1},
t-1-S_{k-2}\giu t-S_k, \\
&& t-S_{k-1}\su t-1-S_k];
\end{array}$$
if $k$ is odd, take
$$\begin{array}{rcl}
rL_1'& = & [0\su t-1-S_1,  t-1-S_0\giu t-S_2, t-S_1\su t-1-S_3,\ldots, \\
&& t-1-S_{k-3} \giu t- S_{k-1}, t-S_{k-2}\su t-1-S_{k},
t-1-S_{k-1}\giu t-S_k].
\end{array}$$
Note that, in all cases, the vertices $2i$ and $2i+1$ are adjacent for all $i=0,\ldots,\frac{t-2}{2}$.
This implies that $rL_1'$ is of type $\C_t$. So, by Proposition \ref{prop:sigmax}(1), we can apply the function $\vartheta_t^{\bar c}$,
obtaining
$rL_1=(\th_t)^{\bar c}(rL_1')$, where $c=2\bar c$.

Hence, for all values of $c\geq 1$, we were able to construct a linear realization of $rL_1$ of type $\C_t$.
So, we can apply $\th_{t-2}, \th_{t-4},\ldots,\th_2$ in this order:
$$ r(L_1\cup L_2)  = ((\th_{2})^{\bar b_2} \circ (\th_{4})^{\bar b_4}\circ \ldots\circ
(\th_{t-2})^{\bar b_{t-2}})(rL_1).$$
Finally, we get $rL=R\{1^k\}+r(L_1\cup L_2)$.
\end{proof}

\begin{ex}
Let $L=\{1^{11}, 2^2, 4^5, 6^6, 8^3, 12^3\}$.
Following the proof of Proposition \ref{all_even}, we write $L=\{1^2\}\cup L_1 \cup L_2$,
where $L_1=\{1^{9}, 4, 8,  12^3 \}$ and $L_2=\{2^2,4^4,6^6,8^2 \}$. Hence,
$rL=R\{1^2\}+r(L_1\cup L_2)$, $x_1=8$, $x_2=4$, $t=12$ and
$$\begin{array}{rcl}
rL'_1=r\{1^{9},4,8, 12\}  & = & [0, 12,11,10,9,  1,2,3,4,  8,7,6,5 ],\\
rL_1=r\{1^{9},4,8, \bm{12^3} \}  & = & [0, 12,11,10,9,  1,\bm{13}, \bm{14}, 2,3,4,  8,7,6,5 ],\\
r\{1^{9},4,\bm{8^3}, 12^3 \}  & = & [0, 12,11,10,9,  1, 13, 14,  2,3,4,  8,\bm{16}, \bm{15}, 7,6,5 ],\\
r\{1^{9},4, \bm{6^6}, 8^3, 12^3 \}  & = & [0, 12,\bm{18}, \bm{17}, 11,10,9,  1, 13, \bm{19}, \bm{20}, 14,  2,3,4,
8,16,\bm{22},\bm{21}, 15,\\
&&7,6,5 ],\\
r\{1^{9},\bm{4^5}, 6^6, 8^3, 12^3 \}  & = & [0, 12, 18, 17, 11,10,9,  1, 13, 19, \bm{23},\bm{24}, 20, 14,  2,3,4,
8,16,22, \\
&& \bm{26},\bm{25}, 21 , 15, 7,6,5 ],\\
r\{1^{9},\bm{2^2}, 4^5, 6^6, 8^3, 12^3 \}  & = & [0, 12, 18, 17, 11,10,9,  1, 13, 19, 23,24, 20, 14,  2,3,4,
8,16,22, 26,\\
&&\bm{28}, \bm{27}, 25, 21 , 15, 7,6,5 ].
\end{array}$$
Now, let $L=\{1^{11},  4^5, 6^6, 8^3, 10^5, 12^4\}$.
Always following the proof of Proposition \ref{all_even}, we write $L=\{1^3\}\cup L_1 \cup L_2$,
where $L_1=\{1^{8}, 4, 8, 10, 12^4 \}$ and $L_2=\{4^4,6^6,8^2,10^4 \}$. Hence,
$rL=R\{1^3\}+r(L_1\cup L_2)$, $x_1=10$, $x_2=8$, $x_3=4$, $t=12$ and
$$\begin{array}{rcl}
rL'_1=r\{1^{8},4,8, 10\}  & = & [0,1,  11,10, 2,3,4,5, 9,8,7,6 ],\\
rL_1=r\{1^{8},4,8, 10, \bm{12^4}\}  & = & [0,\bm{12}, \bm{13}, 1,  11,10, 2,\bm{14}, \bm{15}, 3,4,5, 9,8,7,6 ],\\
r\{1^{8},4,8, \bm{10^5}, 12^4\}  & = & [0,12, 13, 1,  11,10, 2,14, 15, 3,4,5, 9,\bm{19}, \bm{18}, 8,7,\bm{17}, \bm{16},
6 ],\\
r\{1^{8},4,\bm{8^3}, 10^5, 12^4\}  & = & [0,12, \bm{20}, \bm{21}, 13, 1,  11,10, 2,14, 15, 3,4,5, 9, 19, 18,
8,7,17,\\
&& 16,6 ],\\
r\{1^{8},4,\bm{6^6}, 8^3, 10^5, 12^4\}  & = & [0,12, 20, \bm{26}, \bm{27}, 21, 13, 1,  11,10, 2,14, 15, 3,4,5, 9,
19,\bm{25}, \bm{24},\\
&& 18, 8,7,17,\bm{23}, \bm{22}, 16,6 ],\\
r\{1^{8},\bm{4^5},6^6, 8^3, 10^5, 12^4\}  & = & [0,12, 20, 26,\bm{30}, \bm{31}, 27, 21, 13, 1,  11,10, 2,14, 15, 3,4,5,
9, 19,\\
&& 25, \bm{29}, \bm{28}, 24, 18, 8,7,17,23, 22, 16,6 ].
\end{array}$$
\end{ex}

\begin{prop}\label{prop:even+1odd}
Let $x\geq 3$ be an odd integer and let  $1\leq s\leq x-1$.
Then, there exists a standard linear realization of any list
$$\left\{1^{x-1}, 2^{b_2}, 4^{b_4}, 6^{b_6}, \ldots,(x-1)^{b_{x-1}}, x^{s} \right\},$$
where the exponents $b_2,b_4,b_6,\ldots,b_{x-1}$ are all even and nonnegative.
\end{prop}

\begin{proof}
Suppose firstly that $s$ is even and write $s=2\bar s$.
Consider the following linear realization of
$L=\{1^{x-1}\}$: $RL=[0 \su x-1]$.
Since $(0,1), (2,3),\ldots, (x-3,x-2)$ are pairs of adjacent vertices, we
can apply the function
$(\mu_x\circ \eta_x)^{\bar s}$,  obtaining a linear realization of $L'=\{1^{x-1},x^s\}$, whose
vertices $|L'|-(2i+1)=x+s-2i-2$ and $|L'|-2i=x+s-2i-1$ are adjacent for all $i=0,\ldots,\frac{x-3}{2}$.
This means that $(\mu_x\circ \eta_x)^{\bar s}(RL)$ is of type $\C_{x-1}$, and so we get
$$r\{1^{x-1}, 2^{2\bar b_2}, 4^{2 \bar b_4}, \ldots,(x-1)^{2\bar b_{x-1}}, x^{2\bar s} \}=
((\th_{2})^{\bar b_2} \circ (\th_{4})^{\bar b_4} \circ \ldots\circ (\th_{x-1})^{\bar b_{x-1}}\circ
(\mu_x\circ \eta_x)^{\bar s})(RL),$$
where $\bar b_2, \bar b_4, \ldots,\bar b_{x-1}\geq 0$.

Now, suppose that $s$ is odd and write $s=2 \bar s+1$. Start considering the following linear realization of
$\tilde L=\{1^{x-1},x\}$: $r\tilde L=[0,x \giu 1]$.
In this case,  $(1,2), (3,4),\ldots, (x-2,x-1)$ are pairs of adjacent vertices and so we
can apply the function $(\mu_x\circ \eta_x)^{\bar s}$. We then obtain a linear realization of
$\tilde L'=\{1^{x-1},x^s\}$, whose vertices $|\tilde L'|-(2i+1)=x+s-2i-2$ and $|\tilde L'|-2i=x+s-2i-1$ are adjacent
for all $i=0,\ldots,\frac{x-3}{2}$. Now, $(\mu_x\circ \eta_x)^{\bar s}(r\tilde L)=r \tilde L'$ is of type $\C_{x-1}$ and hence we
get
$$r\{1^{x-1}, 2^{2\bar b_2}, 4^{2 \bar b_4}, \ldots,(x-1)^{2\bar b_{x-1}}, x^{2\bar s+1} \}=
((\th_{2})^{\bar b_2} \circ (\th_{4})^{\bar b_4} \circ \ldots\circ (\th_{x-1})^{\bar b_{x-1}}\circ
(\mu_x\circ \eta_x)^{\bar s})(r\tilde L),$$
where $\bar b_2, \bar b_4, \ldots,\bar b_{x-1}\geq 0$.
\end{proof}

\begin{ex}
Consider $L=\{1^8,2^4,6^2,8^6,9^7\}$. According to the notation of Proposition \ref{prop:even+1odd}
we have $y=9$, $s=7$, $\tilde L=\{1^8,9\}$ and $\tilde L'=\{1^8,9^7\}$. Hence $r\tilde L=[0,9\giu 1]$
and
$$\begin{array}{rcl}

r\tilde L'=r\{1^{8}, 9^7\}  & = & [0,9,8,7,6,\bm{15}, \bm{14},5,4, \bm{13}, \bm{12}, 3,2, \bm{11}, \bm{10},1 ],\\
r\{1^{8}, \bm{8^6}, 9^7\}  & = & [0,9,\bm{17}, \bm{16},8,7,6,15,14,5,4,13, \bm{21}, \bm{20}, 12, 3,2, 1, \bm{19}, \bm{18},10, 1 ],\\
r\{1^{8}, \bm{6^2},8^6, 9^7\}  & = & [0,9,17,\bm{23}, \bm{22},16 8,7,6,15,14,5,4,13, 21, 20, 12, 3,2, 1, 19, 18,10, 1 ],\\
r\{1^{8}, \bm{2^4},8^6, 9^7\}  & = & [0,9,17,23, \bm{25}, \bm{27},\bm{26}, \bm{24},16 8,7,6,15,14,5,4,13, 21, 20, 12, 3,2, 1, 19,\\
& & 18,10, 1 ].
\end{array}$$
\end{ex}

\begin{proof}[Proof of Theorem \rm{\ref{pari}}]
Item (1) easily follows from Proposition \ref{all_even}.
For item (2), write $c=qx+s$ with $0\leq s <x$.
By Proposition \ref{all_even} there exists a standard linear realization $\bm{g}$ of 
$\left\{1^{x-2}, 2^{b_2}, 4^{b_4}, 6^{b_6},\ldots, (x-1)^{b_{x-1}} \right\}$.
If $s=0$, by Lemma \ref{perfette}(1) we construct a standard linear realization of $L$ taking
$$rL=R\{1^{a-2x+3}\}+R\{1^{x-1},x^{qx}\}+\bm{g}.$$
If $s>0$, we first apply  Lemmas \ref{odd1} and \ref{odd2} to construct a standard 
linear realization of $\{1^{x-1},x^s\}$.
Then, using Theorem \ref{th:2lr} we get
$$rL=(R\{1^{a-3x+4} \}+R\{1^{x-1},x^{qx} \}+r\{1^{x-1},x^s \})\ol \bm{g}.$$
\end{proof}

\section{On the lists  $\{1^a,2^b,3^c,4^d\}$, $a\geq 2$}\label{new}

Considering the existing results on the $\BHR$ conjecture for lists whose underlying set is $\{1,2,3\}$ or
$\{1,2,3,5\}$, we think it is natural to investigate the lists  $\{1^a, 2^b, 3^c, 4^d \}$.
In particular, the  use of Theorem \ref{th:2lr}
will allow us to almost completely solve this case.
As usual, the more $1$s that are available the easier it is to construct the required realizations.
As we remarked in the Introduction, the cases $L=\{1^a,2^b,3^c\}$ and $L=\{1^a,2^b,4^d\}$ have been completely
solved in \cite{CDF} and in \cite{PPnew}, respectively. So, we may assume $c,d\geq 1$.
The main result is that $\BHR(L)$ holds when $a \geq 3$ and it also holds when $a=2$ provided that $b\geq 1$.
In order to do this we need the following well-known results.

\begin{thm}\cite{CDF}\label{Capp}
A list $\{1^a,2^b,3^c\}$ has a standard linear realization if, and only if, the integers $a,b,c$
satisfy one of the following conditions
\begin{enumerate}
\item[(1)] $a=0$, $b\geq 4$, $c\geq 3$;
\item[(2)] $a=0$, $b=3$ and $c\neq0, 3k+9$ with $k\geq0$;
\item[(3)] $a=0$ and $(b,c)\in \{(2,2),(2,3),(4,1),(4,2),(7,2),(8,2)\}$;
\item[(4)] $a\geq 2$ and  $b=0$;
\item[(5)] $a\geq 1$ and $c=0$;
\item[(6)] $a,b,c\geq 1$ with $(a,b,c)\neq (1,1,3k+5)$ and $k\geq 0$.
\end{enumerate}
\end{thm}

\begin{thm}\cite{PPnew}\label{EJC}
If $a\geq 1$, $b\geq 2$ and $d\geq 0$, the list $\{1^a,2^b,4^d\}$ admits a standard linear realization.
\end{thm}

Now we will see that  Proposition \ref{prop:sigmax} plays a fundamental role in proving the following one.

\begin{prop}\label{M}
The list $\{1^a,2^b,3^c,4^d\}$ admits a standard linear realization in each of the following cases:
\begin{itemize}
  \item [(1)] $a\geq3$, $b\geq0$ even, $0\leq c \leq a-3$ and $d\geq 0$;
  \item [(2)] $a\geq2$, $b\geq1$ odd, $0\leq c \leq a-2$ and $d\geq 0$;
  \item [(3)] $a\geq2$, $b\geq0$ even, $2\leq c \leq a$ and $d\geq 0$ even.
\end{itemize}

\end{prop}
\begin{proof}
(1) We split the proof into $2$ subcases according to the parity of $d$.
\begin{itemize}
\item [(1a)] $d$ even.  Set $d=2d'$ and $b=2b'$.
Let $L=\{1^3\}$ and
note that $RL=[0,1,2,3]$ is a standard linear realization of type $\C_4$.
So if we apply $\th_4^{d'}$ to $rL$, by Proposition \ref{prop:sigmax}(1),  we obtain a standard
linear realization $rL'$ of $L'=L\cup \{4^{2d'}\}=\{1^3,4^d\}$ of type $\C_4$
and hence also of type $\C_2$. Thus by Proposition \ref{prop:sigmax}(2), if we apply $\sigma_2^c$ to $rL'$ we obtain a
standard linear realization $rL''$ of $L''=L'\cup\{1^{c},3^{c}\}=\{1^{c+3},3^c,4^d\}$
of type $\C_2$. Hence we can apply the function $\th_2^{b'}$ to $rL''$. In this way we get a
standard linear realization of $L''\cup \{2^{2b'}\}=\{1^{c+3},2^b,3^c,4^d\}$.
To conclude it is sufficient to consider also the perfect linear realization of $\{1^{a-c-3}\}$.
\item [(1b)] $d$ odd.
 A linear realization of $\{1^a,2^b,3^c,4^d\}$ can be obtained in the same way of case (1a) starting from
$L=\{1^3,4\}$ and
 $rL=[0,4,3,2,1]$ that is a standard linear realization of type $\C_4$.
\end{itemize}

\noindent (2) We split the proof into $2$ subcases according to the parity of $d$.
\begin{itemize}
  \item [(2a)] $d$ even.  Set $d=2d'$ and $b=2b'+1$.
Let $L=\{1^2,2\}$ and
note that $rL=[0,1,3,2]$ is a standard linear realization of type $\C_4$.
  So if we apply $\th_4^{d'}$ to $rL$, by Proposition \ref{prop:sigmax}(1),  we obtain a standard linear realization
  $rL'$ of $L'=L\cup \{4^{2d'}\}=\{1^2,2,4^d\}$ of type $\C_4$
and hence also of type $\C_2$.
Thus by Proposition \ref{prop:sigmax}(2), if we apply $\sigma_2^c$ to $rL'$ we obtain a standard linear realization $rL''$ of $L''=L'\cup\{1^{c},3^{c}\}=\{1^{c+2},2,3^c,4^d\}$
of type $\C_2$. Hence we can apply the function $\th_2^{b'}$ to $rL''$. In this way we get a standard  linear realization
of $L''\cup \{2^{2b'}\}=\{1^{c+2},2^b,3^c,4^d\}$.
To conclude it is sufficient to consider also the perfect linear realization of $\{1^{a-c-2}\}$.
  \item [(2b)] $d$ odd. A linear realization of $\{1^a,2^b,3^c,4^d\}$ can be obtained in the same way of case (2a) starting from
$L=\{1^2,2,4\}$ and
$rL=[0,4,3,1,2]$ that is a standard linear realization of type $\C_4$.
\end{itemize}

\noindent (3) Set $d=2d'$ and $b=2b'$.
Let $L=\{1^2,3^2\}$ and
note that $rL=[0,3,4,1,2]$ is a standard linear realization of type $\C_4$.
So if we apply $\th_4^{d'}$ to $rL$, by Proposition \ref{prop:sigmax}(1),  we obtain a standard linear realization $rL'$ of $L'=L\cup \{4^{2d'}\}$ of type $\C_4$
and hence also of type $\C_2$.
Thus by Proposition \ref{prop:sigmax}(2), if we apply $\sigma_2^{c-2}$ to $rL'$ we obtain a standard
linear realization $rL''$ of $L''=L'\cup\{1^{c-2},3^{c-2}\}$
of type $\C_2$. Hence we can apply the function $\th_2^{b'}$ to $rL''$. In this way we get a standard linear realization of $L''\cup \{2^{2b'}\}=\{1^{c},2^b,3^c,4^d\}$.
To conclude it is sufficient to consider also the perfect linear realization of $\{1^{a-c}\}$.
\end{proof}

\begin{cor}\label{cor5}
Suppose $c,d\geq 1$.
The list $\{1^a, 2^b, 3^c, 4^d\}$ admits a linear realization in each of the following cases:
\begin{itemize}
\item[(1)]  $a\geq 3$ and $b\geq 2$;
\item[(2)]  $a\geq 4$ and $b=1$;
\item[(3)]  $a\geq 5$ and $b=0$.
\end{itemize}
In particular, $\{1^a, 2^b, 3^c, 4^d\}$ admits a  linear realization for all $a\geq 5$.
\end{cor}

\begin{proof}
Applying Theorem \ref{th:2lr}, we can concatenate a standard linear realization of Theorem \ref{Capp} with one of Theorem \ref{EJC} or one of Proposition \ref{M}(2):
$$\begin{array}{rcll}
r\{1^a,2^b,3^c,4^d\}  &=& r\{1^2,3^c\}\ol r\{1^{a-2},2^b,4^d\} & \textrm{if } a\geq 3 \textrm{ and } b\geq 2,\\
r\{1^a, 2^1, 3^c, 4^d\}  & = & r\{1^{a-2}, 3^c \}\ol r\{1^2,2^1,4^d\}& \textrm{if } a\geq 4,\\
r\{1^a, 3^c, 4^d\}  & = & r\{1^{a-3},3^c \}\ol r\{1^3, 4^d\} & \textrm{if } a\geq 5.
\end{array}$$
\end{proof}

\begin{lem}\label{6}
Suppose $b,d\geq 0$ and $c\geq 1$.
The list $\{1^a, 2^b, 3^c, 4^d\}$ admits a standard linear realization in each of the following cases:
\begin{itemize}
\item[(1)] $a\geq 1$, $b=0$, $c\in \{3,4\}$ and $d$ is odd;
\item[(2)] $a\geq 1$, $b=1$ and $c\in \{1,2,3\}$;
\item[(3)] $a\geq 1$, $b=1$, $c\equiv 0,2 \pmod 3$ and $d=1$;
\item[(4)] $a\geq 2$, $b=0$, $c=1$;
\item[(5)] $a\geq 2$, $b=0$, $c=3$ and $d$ is even;
\item[(6)] $a\geq 2$, $b=0$, $c\equiv 0,1 \pmod 3$ and $d=1$;
\item[(7)] $a\geq 3$, $c=2$ and $d$ is odd.
\end{itemize}
\end{lem}

\begin{proof}
First, take the following standard linear realizations,
where $x\arq{y} z$ means the arithmetic progression $x,  x+4, x+8,\ldots,x+4y=z$ if $x\leq z$,
or the arithmetic progression $x,x-4,x-8,\ldots,x-4y=z$ otherwise.

\begin{footnotesize}
$$\begin{array}{l}
      (1) \left\{\begin{array}{rcl}
 r\{1^1, 3^3,4^{4k+1} \}  & = & [0,3\arq{k} 4k+3, 4k+4\arq{k} 4, 1 \arq{k+1} 4k+5, 4k+2 \arq{k} 2],  \\
 r\{1^1, 3^3,4^{4k+3} \}  & = & [0,3\arq{k+1} 4k+7, 4k+4\arq{k} 4, 1 \arq{k+1} 4k+5, 4k+6 \arq{k+1} 2],\\
 r\{1^1, 3^4,4^{4k+1} \}  & = & [ 0,3\arq{k} 4k+3, 4k+6\arq{k+1} 2, 5 \arq{k} 4k+5, 4k+4 \arq{k} 4,1], \\
 r\{1^1, 3^4,4^{4k+3} \}  & = & [ 0,3\arq{k+1} 4k+7, 4k+6\arq{k+1} 2, 5 \arq{k} 4k+5, 4k+8 \arq{k+1} 4,1].
      \end{array}\right.\\[12pt]
   \end{array}$$
$$\begin{array}{l}
         (2) \left\{\begin{array}{rcl}
r\{1^1,2^1,3^1,4^{4k} \} & = & [0\arq{k} 4k, 4k+3\arq{k} 3, 1 \arq{k} 4k+1, 4k+2\arq{k}2 ],\\
r\{1^1,2^1,3^1,4^{4k+1} \} & = & [0\arq{k+1} 4k+4, 4k+1 \arq{k} 1, 3\arq{k} 4k+3, 4k+2\arq{k} 2 ],  \\
r\{1^1,2^1,3^1,4^{4k+2} \} & = & [0\arq{k+1} 4k+4, 4k+3 \arq{k} 3, 1 \arq{k+1} 4k+5, 4k+2 \arq{k} 2  ], \\
r\{1^1,2^1,3^1,4^{4k+3} \} & = & [0\arq{k+1} 4k+4, 4k+5\arq{k+1} 1, 3 \arq{k} 4k+3, 4k+6 \arq{k+1} 2],\\
 r\{1^1, 2^1, 3^2,4^{4k} \}  & = & [0,3\arq{k} 4k+3, 4k+2\arq{k} 2, 4 \arq{k} 4k+4, 4k+1\arq{k} 1],\\
 r\{1^1, 2^1, 3^2,4^{4k+1} \}  & = & [0,3\arq{k} 4k+3, 4k+4\arq{k} 4, 2 \arq{k} 4k+2, 4k+5\arq{k+1} 1],\\
 r\{1^1, 2^1, 3^2,4^{4k+2} \}  & = & [0,3\arq{k} 4k+3, 4k+6\arq{k+1} 2, 4 \arq{k} 4k+4, 4k+5\arq{k+1} 1],\\
 r\{1^1, 2^1, 3^2,4^{4k+3} \}  & = & [0,3\arq{k+1} 4k+7, 4k+4\arq{k} 4, 2 \arq{k+1} 4k+6, 4k+5\arq{k+1} 1],\\
r\{1^1,2^1,3^3,4^{4k} \} & = & [0,3\arq{k} 4k+3, 4k+5, 4k+2\arq{k} 2, 1 \arq{k} 4k+1, 4k+4\arq{k} 4 ],\\
r\{1^1,2^1,3^3,4^{4k+1} \} & = & [0,3\arq{k} 4k+3, 4k+6,  4k+5 \arq{k+1} 1, 4\arq{k} 4k+4, 4k+2\arq{k} 2 ],  \\
r\{1^1,2^1,3^3,4^{4k+2} \} & = & [0,3\arq{k} 4k+3, 4k+6 \arq{k+1} 2, 1 \arq{k+1} 4k+5, 4k+7, 4k+4 \arq{k} 4  ], \\
r\{1^1,2^1,3^3,4^{4k+3} \} & = & [0,3\arq{k+1} 4k+7, 4k+8, 4k+5\arq{k+1} 1, 4 \arq{k} 4k+4, 4k+6 \arq{k+1} 2].
 \end{array}\right.\\
(4) \left\{\begin{array}{rcl}
 r\{1^2, 3^{1},4^{4k} \}  & = & [0\arq{k} 4k,  4k+3 \arq{k} 3,  2 \arq{k} 4k+2, 4k+1 \arq{k} 1],\\
r\{1^2,3^1,4^{4k+1}\} & = & [0\arq{k+1} 4k+4, 4k+1\arq{k} 1, 2\arq{k} 4k+2,4k+3\arq{k} 3],\\
 r\{1^2, 3^{1},4^{4k+2} \}  & = & [0 \arq{k+1} 4k+4,  4k+3 \arq{k} 3,  2 \arq{k} 4k+2,  4k+5\arq{k+1} 1],\\
r\{1^2,3^1,4^{4k+3}\} & = & [0\arq{k+1} 4k+4, 4k+5\arq{k+1} 1, 2\arq{k+1} 4k+6,4k+3\arq{k} 3].
           \end{array}\right.\\[12pt]
(5) \left\{\begin{array}{rcl}
 r\{1^2, 3^{3},4^{4k} \}  & = & [ 0,3 \arq{k} 4k+3, 4k+2 \arq{k} 2,  5 \arq{k} 4k+5, 4k+4 \arq{k} 4,1 ],\\
 r\{1^2, 3^{3},4^{4k+2} \}  & = & [0,3\arq{k+1} 4k+7, 4k+6\arq{k+1} 2,  5 \arq{k} 4k+5, 4k+4 \arq{k} 4,1 ].
 \end{array}\right.\\[12pt]
 (7) \left\{\begin{array}{rcl}
 r\{1^3, 3^2,4^{4k+1} \}  & = & [0\arq{k} 4k, 4k+3\arq{k} 3, 2\arq{k+1} 4k+6, 4k+5, 4k+4, 4k+1\arq{k} 1],\\
 r\{1^3, 3^2,4^{4k+3} \}  & = & [0\arq{k+1} 4k+4, 4k+1\arq{k} 1, 2\arq{k+1} 4k+6, 4k+5, 4k+8, 4k+7\arq{k+1}3].
  \end{array}\right.\\
  \end{array}$$
\end{footnotesize}

To prove (3) and (6) take the following standard linear realizations.
Here $x\art{y} z$ means the arithmetic progression $x,  x+3, x+6,\ldots,x+3y=z$ if $x\leq z$,
or the arithmetic progression $x,x-3,x-6,\ldots,x-3y=z$ otherwise.
$$\begin{array}{l}
(3) \left\{\begin{array}{rcl}
 r\{1^1,2^1, 3^{3k+2},4^1 \}  & = & [0\art{k} 3k, 3k+4 \art{k+1} 1, 2 \art{k+1}  3k+5, 3k+3],\\
 r\{1^1,2^1, 3^{3k+3},4^1 \}  & = & [0\art{k} 3k, 3k+4 \art{k+1} 1, 2 \art{k+1}  3k+5, 3k+3\art{1} 3k+6].
 \end{array}\right.\\[12pt]
 (6) \left\{\begin{array}{rcl}
 r\{1^2, 3^{3k+1},4^1 \}  & = & [0\art{k} 3k, 3k+4 \art{k+1} 1, 2 \art{k}  3k+2, 3k+3],\\
 r\{1^2, 3^{3k+3},4^1 \}  & = & [0\art{k+2} 3k+6, 3k+2 \art{k} 2, 1 \art{k+1}  3k+4, 3k+5].
 \end{array}\right.
  \end{array}$$
The statement follows by adding in each case $R\{1^{\tilde a}\}$ for all $\tilde a\geq 0$.
   \end{proof}

\begin{prop}\label{a34}
Suppose $c,d\geq 1$.
The lists $\{1^3, 2^b, 3^c, 4^d\}$ and $\{1^4, 2^b, 3^c, 4^d\}$ admit a  linear realization for all $b\geq 0$.
\end{prop}

\begin{proof}
Assume $a=3$: by Corollary \ref{cor5} we are left to the case $b\in \{0,1\}$.
If $b=1$, by Theorem \ref{Capp}(4) and Lemma \ref{6}(2) we have
$$r\{1^3, 2^1, 3^c, 4^d\}=r\{1^2, 3^{c-1}\}\ol r\{1^1, 2^1, 3^1, 4^d\}.$$
Suppose $b=0$.  If $c=1$, apply Lemma \ref{6}(4);
if $c=2$, apply  Lemma \ref{6}(7)  when $d$ is odd and Proposition \ref{M}(3)  when $d$ is even.
If $c\geq 3$ and $d$ is odd, by Theorem \ref{Capp}(4) and Lemma \ref{6}(1) we have
$$r\{1^3, 3^c, 4^d\}=r\{1^2, 3^{c-3}\} \ol r\{1^1, 3^3, 4^d\}.$$
If $c=3$ and $d$ is even, apply  Lemma \ref{6}(5).
If $c\geq 4$ and $d$ is even,  by Lemma \ref{6}, cases (1) and (6), we have
$$\begin{array}{rcl}
r\{1^3, 3^{3k+4}, 4^d\} & = & r\{1^2, 3^{3k+1},4^1\} \ol r\{1^1, 3^3, 4^{d-1}\},\\
r\{1^3, 3^{3k+5}, 4^d\} & = & r\{1^2, 3^{3k+1},4^1\} \ol r\{1^1, 3^4, 4^{d-1}\},\\
r\{1^3, 3^{3k+6}, 4^d\} & = & r\{1^2, 3^{3k+3},4^1\} \ol r\{1^1, 3^3, 4^{d-1}\}.
  \end{array}$$

Now, assume $a=4$: by Corollary \ref{cor5} we are left to the case $b=0$.
We can use Theorem \ref{Capp}(4) and Lemma \ref{6}(4):
$$r\{1^4,  3^c, 4^d\} =  r\{1^2, 3^{c-1}\} \ol r\{1^2, 3^1, 4^d\}.$$
\end{proof}

\begin{prop}\label{a2}
Suppose $c,d\geq 1$.
The list $\{1^2, 2^b, 3^c, 4^d\}$ admits a  linear realization for all $b\geq 1$.
\end{prop}

\begin{proof}
If $b\geq 2$, then we use a standard linear realization of Theorem \ref{Capp}(6) and one of Lemma~\ref{6}(2):
$$r\{1^2, 2^b, 3^{3k+\bar c} , 4^d\}=r\{1^1, 2^{b-1}, 3^{3k}\} \ol r\{1^1,2^1,3^{\bar c}, 4^d\},$$
where $\bar c=1,2,3$.

Assume now $b=1$. If $c\in \{1,2,3\}$, the result follows directly from Lemma \ref{6}(2).
So we may suppose $c\geq 4$.
If $d$ is odd, we use  a standard linear realization of Theorem \ref{Capp}(6) and one of Lemma \ref{6}(1):
$$\begin{array}{rcl}
r\{1^2,2^1, 3^{3k+4}, 4^d\} & = & r\{1^1,2^1, 3^{3k}\} \ol r\{1^1, 3^4, 4^{d}\},\\
r\{1^2,2^1, 3^{3k+5}, 4^d\} & = & r\{1^1,2^1, 3^{3k+1}\} \ol r\{1^1, 3^4, 4^{d}\},\\
r\{1^2,2^1, 3^{3k+6}, 4^d\} & = & r\{1^1,2^1, 3^{3k+3}\} \ol r\{1^1, 3^3, 4^{d}\}.
  \end{array}$$
Finally, if $d$ is even,
we use standard  realizations of Lemma \ref{6}, cases (1) and (3):
$$\begin{array}{rcl}
r\{1^2,2^1, 3^{3k+4}, 4^d\} & = & r\{1^1,2^1, 3^{3k},4^1\} \ol r\{1^1, 3^4, 4^{d-1}\},\\
r\{1^2,2^1, 3^{3k+5}, 4^d\} & = & r\{1^1,2^1, 3^{3k+2},4^1\} \ol r\{1^1, 3^3, 4^{d-1}\},\\
r\{1^2,2^1, 3^{3k+6}, 4^d\} & = & r\{1^1,2^1, 3^{3k+2},4^1\} \ol r\{1^1, 3^4, 4^{d-1}\}.
  \end{array}$$
\end{proof}

Note that in order to give a complete solution also for $a=2$ it remains to consider the lists
$\{1^2, 3^c, 4^d\}$ with $c=2$ or $c\geq 4$.

\begin{proof}[Proof of Theorem \rm{\ref{th1234}}]
The three cases $a=2$, $a\in \{3,4\}$ and $a\geq 5$ follow, respectively, from
Proposition \ref{a2}, Proposition \ref{a34} and Corollary \ref{cor5}.
\end{proof}

We remark that the cases $a=0,1$ cannot be solved using only linear realizations since,
for instance, there is no linear realization of the lists $\{2,3,4^6\}$ and $\{1,2,4^6\}$ (see \cite[Lemma 13]{PPnew}).

A similar result to Theorem \ref{th1234} can be obtained also for lists $\{1^a,2^b,3^c,4^d,5^e,6^f\}$.
In this case, we recall the following.

\begin{thm}\cite[Theorem 2.6]{PP1}\label{1235}
If  $a\geq 5$ the list $\{1^a, 2^b, 3^c, 5^e\}$ admits a standard linear realization.
\end{thm}

\begin{prop}
The following lists $L$ admit a linear realization:
\begin{itemize}
\item[(1)] $L=\{1^a, 2^b, 3^c, 4^d, 5^e\}$ with $a \geq 8$ and $e\geq 1$;
\item[(2)] $L=\{1^a, 2^b, 3^c, 4^d, 5^e, 6^f\}$ with $a \geq 10$ and $f\geq 1$.
\end{itemize}
\end{prop}

\begin{proof}
Concatenate a standard linear realization of Proposition \ref{all_even} with one of Theorem \ref{1235}:
$$\begin{array}{ll}
r\{1^a,2^b,3^c,4^d,5^e\}=r\{1^{a-5},4^d\} \ol r\{1^5,2^b,3^c,5^e\} & \textrm{if } a\geq 8 \textrm{ and } e\geq 1,\\
r\{1^a,2^b,3^c,4^d,5^e,6^f\}=r\{1^{a-5}, 4^d, 6^f\}\ol r\{1^5,2^b,3^c,5^e\} & \textrm{if } a\geq 10 \textrm{ and } f\geq 1.
  \end{array}$$
\end{proof}

\section{On the lists with many distinct elements}\label{sec:small}

In this section we give constructions of linear realizations for lists that are ``close" to $L^*$,
in a way we make precise below, and also for lists of the form $\{ 1^{a_1}, \ldots, m^{a_m} \}$ with $a_1 \geq a_2 \geq \cdots \geq a_m$.
We recall that  $L^* = \{ 1^2, 2^2, \ldots, m^2 \}$ if $v = 2m+1$, and $L^* = \{ 1^2, 2^2, \ldots, (m-1)^2, m^1 \}$ if $v=2m$.
Graceful permutations are central to both constructions and we make heavy use of the following result:

\begin{lem}\label{lem:g1}{\rm \cite{Kotzig73}}
For all~$v$ and all~$x$ with  $0 \leq x < v$, there is a graceful permutation $[g_1, \ldots, g_v]$ with $g_1=x$.
\end{lem}

Horak and Rosa~\cite{HR} introduced a graph~$\mathcal{G}_v$ for prime~$v$ where the vertices are lists of size~$v-1$
with underlying set a subset of~$\{ 1, \ldots, (v-1)/2 \}$.  Two vertices $L$ and  $M$ are adjacent in~$\mathcal{G}_v$ if
$M = L \setminus \{ x \} \cup \{y\}$ for some~$x \neq y$.  Buratti's Conjecture is equivalent to saying that all
vertices of this graph are realizable.  Horak and Rosa showed that $L^*$ and its neighbors are realizable.

If $v$ is composite we make a similar definition for $\mathcal{G}_v$.  The only difference is that we now
insist that the vertex labels meet the necessary conditions of the $\BHR$ Problem, and so Conjecture
\ref{HR iff B} is equivalent to saying that all vertices of this graph are realizable.

Theorems~\ref{th:gkh} and~\ref{th:gmm} cover many instances of the conjecture that are distance~$2$ from $L^*$ in $\mathcal{G}_v$.

\begin{thm}\label{th:gkh}
Let~$v = 2m$ or $2m+1$ and let $1 \leq k < m$.
$\BHR(L)$ holds when  $L = L^* \setminus \{k,k+1\}   \cup \{z_1,z_2 \}$ for any $1 \leq z_1 \leq k$ and $1 \leq z_2 \leq m$.
\end{thm}

\begin{proof}
Let $\bm{g} = [g_1, \ldots, g_k]$ be a graceful permutation of length~$k$ with $g_k = k-z_1$, which exists by Lemma~\ref{lem:g1}.
Let $\bm{h} = [h_1, \ldots, h_{v-k-1}]$ be a graceful permutation of length~$v-k-1$ with~$h_1 = z_2-1$, which again exists by
Lemma~\ref{lem:g1}.  The differences of~$\bm{h}$ in $\Z_v$ are
$\{ 1, \ldots, k+1, (k+2)^2, \ldots, m^2 \}$ if $v$ is odd and $\{ 1, \ldots, k+1, (k+2)^2, \ldots, (m-1)^2, m \}$ if $v$ is even.

The concatenation of $\bm{g}, [k]$, and $\bm{h} + k+1$ has the differences
$$ \{ 1, \ldots, k-1 \} \cup \{ k-g_k , h_1+k+1-k  \} \cup \{ 1, \ldots, k+1, (k+2)^2, \ldots, m^2 \} $$
when $v$ is odd, and
$$ \{ 1, \ldots, k-1 \} \cup \{ k-g_k , h_1+k+1-k  \} \cup \{ 1, \ldots, k+1, (k+2)^2, \ldots, (m-1)^2, m \} $$
when $v$ is even.
Now, $k-g_k = z_1$ and $h_1+k+1-k = z_2$, completing the proof.
\end{proof}

\begin{ex}\label{exa:gkh}
In the notation of Theorem~\ref{th:gkh}, let $v = 2m+1 = 17$, $k=6$ and $(z_1,z_2) = (2,4)$.
Take $\bm{g} = [3,2,0,5,1,4]$, a graceful permutation of length $k=6$ with final element $k-z_1 = 4$.
Take $\bm{h} = [3,6,0,9,1,8,4,5,7,2]$, a graceful permutation of length $v-k-1 = 10$ with first element $z_2 - 1 = 3$.
Then the proof of Theorem~\ref{th:gkh} gives the realization
$$[3,2,0,5,1,4,6,10,13,7,16,8,15,11,12,14,9]$$
of $\{ 1^2, \ldots, 8^2 \} \setminus \{6,7\} \cup \{2,4\}$.
\end{ex}

A variation of the construction works for some cases with the removed elements equal to each other:

\begin{thm}\label{th:gmm}
If~$v = 2m+1$ then $\BHR(L)$ holds when $L = L^* \setminus \{m,m\}   \cup \{z_1,z_2 \}$ for any  $1 \leq z_1,  z_2 < m$.
If $v = 2m$ then $\BHR(L)$ holds when $L = L^* \setminus \{m-1,m-1\}   \cup \{z_1,z_2 \}$ for any  $1 \leq z_1,  z_2 < m-1$.
\end{thm}

\begin{proof}
Let $v = 2m+1$ and $1 \leq z_1,  z_2 < m$.  Let $\bm{g} = [g_1, \ldots, g_m]$ be a graceful permutation of length~$m$ with $g_m = m-z_1$.
Let $\bm{h} = [h_1, \ldots, h_{m}]$ be a graceful permutation of length~$m$ with~$h_1 = z_2-1$.
Then the concatenation of $\bm{g}$, $[m]$ and  $\bm{h} + m+1$ is the required sequence as we have the differences
$$ \{ 1, \ldots, m-1 \} \cup \{ m-g_m , (h_1+m+1) -m  \} \cup \{ 1, \ldots, m-1 \},$$
and $m-g_m = z_1$ and $(h_1+m+1)-m = z_2$.

Let $v = 2m$ and  $1 \leq z_1,  z_2 < m-1$.  Let $\bm{g} = [g_1, \ldots, g_{m-1}]$ be a graceful permutation of length~$m-1$
with $g_{m-1} = m-1-z_1$.  Let $\bm{h} = [h_1, \ldots, h_{m-1}]$ be a graceful permutation of length~$m-1$ with~$h_1 = m-1-z_2$.
Then the concatenation of $\bm{g}$, $[m-1, 2m-1]$ and  $\bm{h} + m$ is the required sequence as we have the differences
$$ \{ 1, \ldots, m-2 \} \cup \{ m-1-g_{m-1} , m, 2m-1 - (h_1+m)  \} \cup \{ 1, \ldots, m-2 \},$$
and $m-1-g_m = z_1$ and $2m-1 - (h_1+m) = z_2$.
\end{proof}

\begin{ex}
In the notation of Theorem~\ref{th:gmm}, let $v = 2m+1 = 21$ and $(z_1, z_2) = (6,4)$.
Take $\bm{g} = [5,3,8,0,9,2,6,7,1,4]$, a graceful permutation of length $m=10$ with final element $m-z_1 = 4$.
Take $\bm{h} = [3,6,0,9,1,8,4,5,7,2]$, a graceful permutation of length $m = 10$ with first element $z_2 - 1 = 3$.
Then the proof of Theorem~\ref{th:gkh} gives the realization
$$[5,3,8,0,9,2,6,7,1,4, 10, 14,17,11,20,12,19,15,16,18,13   ]$$
of $\{ 1^2, \ldots, 10^2 \} \setminus \{10,10\} \cup \{6,4\}$.
\end{ex}

While the results of Theorems~\ref{th:gkh} and~\ref{th:gmm} may seem narrow as stated, we can use the same idea described at the end of
Section~\ref{sec:xodd} to extend their scope. For example, considering the realization in Example~\ref{exa:gkh} and multiplying by $2$
(modulo~$17$), we see that $\BHR(L)$ holds for $L = \{ 1^2, \ldots, 8^2 \} \setminus \{3,5\} \cup \{4,8\}$.
With this method we cover a large number of cases at distance~$2$ from~$L^*$ in~$\mathcal{G}_v$.

Theorems~\ref{th:gkh} and~\ref{th:gmm} were proved by concatenating translations of graceful permutations while
using Lemma~\ref{lem:g1} to control the differences introduced at the joins.  The same is true for Theorem~\ref{th:desc},
which easily implies Theorem \ref{66}.

\begin{thm}\label{th:desc}
Let $L = \{ 1^{a_1}, \ldots, m^{a_m} \}$ with $a_1 \geq a_2 \geq \cdots \geq a_m > 0$.  Then~$L$ has a standard linear realization.
\end{thm}

\begin{proof}
Write
\begin{eqnarray*}
L &=& \{ 1, \ldots, k_1\} \cup \{1, \ldots, k_2 \} \cup \cdots \cup \{ 1, \ldots, k_q \} \\
  & = & \{ 1, \ldots, k_1 -1 \} \cup \{1, \ldots, k_2 -1 \} \cup \cdots \cup \{ 1, \ldots, k_{q-1} -1 \} \cup\\
  &&\{ 1, \ldots, k_q \} \cup \{ k_1, \ldots, k_{q-1} \}
\end{eqnarray*}
with $k_1 \leq \cdots \leq k_q = m$.
Set~$L_i = \{ 1, \ldots, k_i -1 \}$ for~$1 \leq i < q$ and $L_q = \{ 1, \ldots, k_q \}$.

For $1 \leq i \leq q-1$, let $\bm{g_i} = [g_{i,1}, \ldots, g_{i,k_i }]$ be a graceful permutation of length~$k_i $
and let $\bm{g_q} = [g_{q,1}, \ldots, g_{q,k_q+1}] $ be a graceful permutation of length~$k_q+1$.  Construct these so that $g_{1,1} = 0$ and $g_{i+1,1} = g_{i,k_i}$ for each~$i$, which is possible by Lemma~\ref{lem:g1}.

Consider the sequence obtained by concatenating the following sequences:
$$ \bm{g_1}, \ \bm{ g_2} + k_1 , \  \bm{g_3} + k_1 + k_2 , \ldots, \ \bm{ g_{q}} +  k_1 + \cdots + k_{q-1}.$$
Internally to the translations of linear realizations we get the differences
$  L_1 \cup \cdots \cup L_q $ as each $\bm{g_i}$ is a graceful permutation.
At the $i$th join we have
$$(g_{i+1,1} + k_1 + \cdots + k_i) - (g_{i,k_i} + k_1 + \cdots + k_{i-1}) = k_i$$
as $g_{i+1,1} = g_{i,k_i}$, and so the remaining differences from our sequence are $\{ k_1, \ldots, k_{q-1} \}$.  Hence the sequence is a linear realization for~$L$.  As $g_{1,1} = 0$, it is standard.
\end{proof}

\begin{ex}\label{ex:v19}
Suppose $v=19$ and $L = \{ 1^6, 2^3, 3^3, 4^2, 5^2, 6^2 \}$.  In the notation of the proof of Theorem~\ref{th:desc}, $q=6$ and
$(k_1, \ldots, k_6) = (1,1,1,3,6,6)$.  Using the graceful permutations
$$[0], [0], [0], [0,2,1], [1,5,0,3,2,4], [4,2,3,6,0,5,1],$$
results in the sequence
$$[0,1,2,3,5,4,7,11,6,9,8,10,16,14,15,18,12,17,13]$$
to satisfy $\BHR(L)$.
\end{ex}

As the realization in Theorem~\ref{th:desc} is standard, we may combine it with other standard realizations using Theorem~\ref{th:2lr}.
For example, we can remove the restriction that $a_1 \geq a_2$ provided that $a_1 > a_3$.

\begin{cor}\label{cor:desc2}
Let $L = \{ 1^{a_1}, \ldots, m^{a_m} \}$ with $a_2 \geq a_3 \geq \cdots \geq a_m$ and $a_1 > a_3$.  Then~$L$ has a linear realization.
\end{cor}

\begin{proof}
Construct a standard linear realization for
$$L_1 = \{  1^{a_1 - 1},  2^{a_3}, 3^{a_3},  \ldots, m^{a_m} \}$$
using Theorem~\ref{th:desc}.  Construct a standard linear realization for $L_2 = \{ 1, 2^{a_2-a_3} \}$ using Theorem~\ref{Capp}(5).
Apply Theorem~\ref{th:2lr} to obtain a linear realization for $L_1 \cup L_2 = L$.
\end{proof}

\section*{Acknowledgements}

The second and the third author were partially supported by INdAM--GNSAGA.

\end{document}